%% file: main.tex
\documentclass[11pt]{amsart}

\usepackage[margin=1in]{geometry}

\usepackage[ansinew]{inputenc}

\usepackage{textcomp}

\usepackage{amsmath,amsfonts,amssymb,amsthm,amscd,mathrsfs}
\usepackage{bbm}
\usepackage{tikz}
\usetikzlibrary{arrows}

\makeatletter
\tikzset{my loop/.style =  {to path={
  \pgfextra{}
  [looseness=20,min distance=2cm]
  \tikz@to@curve@path},font=\sffamily\small
  }}  
\makeatletter 
\tikzstyle{loop left}= [out=100, in=150, my loop]

\usepackage[all]{xypic}
\usepackage{url}
\usepackage{epsfig}
\usepackage{longtable}

\usepackage{color}
\usepackage{array}
\usepackage{comment}

\usepackage[
alphabetic,
msc-links,
nobysame,
lite,
]{amsrefs} % for bibliography

\usepackage{setspace}

\usepackage{algorithm} % has to be loaded AFTER hypperrf
\usepackage{algorithmic}

\newtheorem{thm}{Theorem}[section]

\newtheorem{lem}[thm]{Lemma}
\newtheorem{prop}[thm]{Proposition}

\newtheorem{remark}[thm]{Remark}

\theoremstyle{definition}
\newtheorem{defn}[thm]{Definition}%[section]

\theoremstyle{definition}
\newtheorem{example}[thm]{Example}

\newcommand{\Z}{\mathbb{Z}}
\newcommand{\R}{\mathbb{R}}

\newcommand{\RR}{\mathbb{R}}

\newcommand{\cG}{\mathcal{G}}
\newcommand{\cF}{\mathcal{F}}
\newcommand{\cP}{\mathcal{P}}

\newcommand{\conv}{\textrm{conv}}
\newcommand{\rank}{\textrm{rank}}
\newcommand{\pa}{\textrm{pa}}
\newcommand{\an}{\textrm{an}}
\renewcommand{\emptyset}{\varnothing}

\DeclareMathOperator{\Span}{span}

 \newcommand\independent{\protect\mathpalette{\protect\independenT}{\perp}}
    \def\independenT#1#2{\mathrel{\rlap{$#1#2$}\mkern2mu{#1#2}}}
\newcommand\ind{\independent}
\newcommand\notindependent{\!\perp\!\!\!\!\not\perp\!}

\DeclareMathOperator{\cone}{cone}
\DeclareMathOperator{\face}{face}

\begin{document}

\title[Generalized Permutohedra from Probabilistic Graphical Models]{Generalized Permutohedra \\ from Probabilistic Graphical Models}
\author{Fatemeh Mohammadi}
\address{ School of Mathematics, University of Bristol,  Bristol, BS8 1TW, UK}
\email{fatemeh.mohammadi@bristol.ac.uk}
\author{Caroline Uhler}
\address{Department of Electrical Engineering \& Computer Science, and Institute for Data, Systems and Society, Massachusetts Institute of Technology, Cambridge MA, USA}
\email{cuhler@mit.edu}
\author{Charles Wang}
\address{Department of Mathematics, University of California, Berkeley, CA, USA}
\email{ charles@math.berkeley.edu}
\author{Josephine Yu}
\address{School of Mathematics, Georgia Institute of Technology,
        Atlanta GA, USA}
\email {jyu@math.gatech.edu}

\date{}

\begin{abstract}
A graphical model encodes conditional independence relations via the Markov properties. For an undirected graph these conditional independence relations can be represented by a simple polytope known as the graph associahedron, which can be constructed as a Minkowski sum of standard simplices. There is an analogous polytope for conditional independence relations coming from a regular Gaussian model, and it can be defined using multiinformation or relative entropy.  For directed acyclic graphical models and also for mixed graphical models containing undirected, directed and bidirected edges,  we give a construction of this polytope, up to equivalence of normal fans, as a Minkowski sum of matroid polytopes.  Finally, we apply this geometric insight to construct a new ordering-based search algorithm for causal inference via directed acyclic graphical models.
\end{abstract}
\maketitle

\input{1_introduction.tex}

\input{2_background.tex}

\input{3_BN.tex}

\input{4_multiinformation.tex}

\input{5_MSMP.tex}

\input{6_mixedGraphs.tex}

\input{7_relations.tex}

\input{8_faces.tex}

\appendix
\input{definitions.tex}

\input{submodular_proof.tex}

\input{dictionary.tex}

\section*{Acknowledgment}
CU was partially supported by DARPA (W911NF-16-1-0551), NSF (DMS-1651995), ONR (N00014-17-1-2147), a Sloan Fellowship, and the Austrian Science Fund (Y 903-N35).  JY was partially supported by the US NSF grant DMS \#1600569. We are grateful to the anonymous referees for very helpful comments on earlier versions of this paper.

\bibliography{DAG}
\bibliographystyle{beta}
\end{document}

%% file: 1_introduction.tex
\section{Introduction}

A graphical model encodes conditional independence (CI) relations via the Markov properties. Our main goal is to understand the polyhedral geometry and combinatorics of the collection of CI relations encoded by a directed acyclic graph (DAG), a directed graph without directed cycles.  It is natural, especially in view of causal inference, to associate to each conditional independence statement a collection of pairs of adjacent permutations of random variables that are compatible with that statement. Each of these pairs can be viewed as an edge of a permutohedron or a wall in the $S_n$ fan, which is the normal fan of the permutohedron.  Removing these walls gives a coarsening of the fan and a natural question is whether this fan is the normal fan of a polytope. 

For undirected graphical models, the theory is well understood.  The coarsening of the $S_n$ fan corresponding to the CI relations encoded by an undirected graph is the normal fan of a polytope called a \emph{graph associahedron}~\cite{Morton_et_al}.  These polytopes are Minkowski sums of standard simplices (MSS), and their facial structure has a nice description via {\em tubings}~\cites{Carr_Devadoss, PostnikovReinerWilliams}. 

In this paper we will show that the coarsened $S_n$ fan arising from any DAG model is the normal fan of a polytope, which we call a \emph{DAG associahedron}.  We give two concrete constructions of DAG associahedra, one using multiinformation or relative entropy, and another using matroids. While in this paper we mainly concentrate on DAG models, we also show that these two constructions can be extended to more general graphical models that have been studied in the literature containing a mix of undirected, directed and bidirected edges. In contrast to graph associahedra, we show that DAG associahedra are in general not simple polytopes and cannot be realized as a Minkowski sum of standard simplices. Our main motivation for studying DAG associahedra is causal inference: Given a set of CI relations that are inferred from data, the goal is to estimate the underlying DAG model, also known as a Bayesian network. A DAG is defined by an ordering of the nodes and an undirected graph. We show how our geometric insight on DAG associahedra can be applied to construct a new ordering-based search algorithm for causal inference. 

Other polyhedral approaches for learning Bayesian networks have been described in the literature~\cite{Cussens_2016, hemmecke2012characteristic, Jaakkola_2010, Studeny_2010, Studeny_2012}. These approaches are based on using integer programming or linear programming relaxations to maximize a score function over a polytope --- most notably, the Family Variable Polytope (FVP) and the Characteristic Imset Polytope (CIP), whose vertices correspond to all possible DAGs on $n$ nodes, up to Markov equivalence, respectively. While the FVP and the CIP are high-dimensional ($n(2^{n-1}-1)$ and $(2^{n}-n-1)$, respectively) and very complex polytopes (facet description only known for $n\leq 4$), we here present a new polyhedral approach for learning Bayesian networks that is based on DAG associahedra, $n-1$-dimensional polytopes for which we give a concrete construction.

%% file: 2_background.tex
\section{Notation and background}
\label{sec:background}

In this section, we discuss the relationship between CI relations, the $S_n$ fan, and generalized permutohedra.  We refer the reader to Appendix~\ref{sec:def} for basic definitions concerning polytopes and fans and to Appendix~\ref{sec:dictionary} for a ``dictionary'' of concepts.

%Please refer to Appendix~\ref{sec:def} for basic definitions of polytopes and fans and Appendix~\ref{sec:dictionary} for a ``dictionary'' of concepts.

Let $[n]=\{1,\dots, n\}$, and let $\mathbb{P}$ be a joint distribution on the random variables $X_i$ for $i\in [n]$. For notational simplicity we often write $I$ for $\{X_i: i\in I\}$ where $I\subseteq [n]$. For pairwise disjoint subsets $I,J,K\subset [n]$ we say that $I$ is \emph{conditionally independent} of $J$ given $K$ under $\mathbb{P}$ if the conditional probability $\mathbb{P}(\mathcal{A}\mid J,K)$ does not depend on $J$ for any measurable set $\mathcal{A}$ in the sample space of~$X_I$. This statement is denoted by $I\independent_\mathbb{P} J\mid K$ or simply $I\independent J\mid K$. If $K=\emptyset$, we write $I\independent J$. The set of CI relations arising from a distribution satisfies the following basic implications, known as the \emph{semigraphoid} properties~\cite{Pearl}:
\begin{enumerate}
\item[(SG1')] if $I\independent J\mid L$ then $J\independent I\mid L$,
\item[(SG2')] if $I\independent J\mid L$ and $U\subseteq I$, then $U\independent J\mid L$,
\item[(SG3')] if $I\independent J\mid L$ and $U\subseteq I$, then $I\setminus U\independent J\mid (U\cup L)$,
\item[(SG4')] if $I\independent J\mid L$ and  $I\independent K\mid J\cup L$, then $I\independent (J\cup K)\mid L$.
\end{enumerate}

In this paper, CI relations can be considered as formal constructs and do not necessarily have probabilistic interpretation.  In addition, we will only work with relations in which $I$ and $J$ are both singletons, denoted by lowercase letters $i, j$; see~\cite{Matus_1992}. To simplify notation, we use concatenation to denote union among subsets and elements of $[n]$, e.g.\ $Lij$ means $L \cup \{i,j\}$. Then a {\em semigraphoid} can be identified with a set of elementary CI relations
\begin{enumerate}
\item[(SG1)] if $i\independent j\mid L$ then $j\independent i\mid L$,
\item[(SG2)] if $i\independent j\mid L$ and  $i\independent k\mid jL$, then $i\independent k\mid L$ and $i\independent j\mid kL$,
\end{enumerate}
for distinct $i,j,k\in[n]$ and $L\subseteq [n]\setminus\{i,j,k\}$.

For distributions with strictly positive densities such as regular Gaussian distributions, the \emph{intersection axiom} holds in addition to the semigraphoid axioms, namely
\begin{enumerate}
\item[(INT)] if $i\independent j\mid kL$ and  $i\independent k\mid jL$, then $i\independent j \mid L$ and $i\independent k\mid L$.
\end{enumerate}
The implications (SG1), (SG2) and (INT) together are known as the \emph{graphoid} properties. Note that these implications are not a complete list of CI implications that hold for distributions. In fact, Studen\'y~\cite{Studeny92} proved that there exists no finite such characterization. 

In~\cite{Matus_Gaussoids}, Ln\v{e}ni\v{c}ka and Mat\'u\v{s} defined {\em gaussoids} as the graphoids satisfying the following additional axioms: 
\begin{enumerate}
\item[(G1)] if $i \independent j \mid L$ and $i \independent k \mid L$, then $i \independent j \mid kL$ and $i \independent k \mid jL$,
\item[(G2)] if $i \independent j \mid L$ and $i \independent j \mid kL$, then $i \independent k \mid L$ or $j \independent k \mid L$.
\end{enumerate}
The property (G1) is the converse of the intersection axiom, and (G2) is known as {\em weak transitivity}.  The CI relations of any regular Gaussian distribution form a gaussoid, but not all gaussoids arise this way.  The set of CI relations coming from any undirected graphical model or a DAG model can be faithfully represented by a regular Gaussian distribution, hence forming a gaussoid.

%In this paper, CI relations can be considered as a formal construct and do not necessarily need any probabilistic interpretation.  However we will see later that the DAG semigraphoids that we study can be faithfully realized by regular Gaussian distributions.

We will associate a geometric object to a collection of CI relations as follows: Consider the hyperplanes in $\R^n$ defined by equations of the form $x_i = x_j$ for all $1 \leq i <j \leq n$.  The complement of these hyperplanes consists of points in $\R^n$ with distinct coordinates, and they are partitioned into $n!$ connected components corresponding to the permutations of $[n]$ as follows:  We identify a permutation (bijection) $\pi: [n] \rightarrow [n]$ with the linear order $\pi(1) \succ \pi(2) \succ \cdots \succ \pi(n)$.  To every vector $u \in \R^n$ with distinct coordinates, we associate a linear order $\succ$ on $[n]$ by defining $i \succ j$ if and only if $u_i > u_j$.  For example, the vector $u = (25,4,16,9)$ gives the linear order $1 \succ 3 \succ 4 \succ 2$, which we denote using its {\em descent vector} of the form $(1|3|4|2)$.  Two points in the complement of the hyperplanes $x_i=x_j$ in $\RR^n$ are in the same connected component if and only if they have the same descent vector.

The closures of the $n!$ cones and all their faces form a fan, which we will call the {\em $S_n$ fan}. It is also known as the {\em permutohedral fan} or the {\em $A_{n-1}$ fan} or the {\em braid arrangement fan}.   Each cone in the fan contains the line in direction $(1,1,\dots,1)$ and is generated by a collection of $0/1$ vectors, every pair of which is nested (when each $0/1$ vector is identified with its set of nonzero coordinates). 

To each CI relation $i \independent j \mid K$, where $i,j \in [n]$ distinct and $K\subseteq [n]\setminus\{i,j\}$, we associate pairs of adjacent permutations of the form 
\begin{equation} 
\label{eqn:adjacent}
(a_1|\cdots|a_k|i|j| b_1|\cdots|b_{n-k-2}) \text{ and } (a_1|\cdots|a_k|j|i| b_1|\cdots|b_{n-k-2}),
\end{equation}
where $\{a_1,\dots,a_k\} = K$ and $\{b_1,\dots,b_{n-k-2}\} = [n]\backslash (K\cup\{i,j\})$.    % The union of the two linear ordering relations is a preorder denoted 
We will denote such a pair by $(a_1|\cdots|a_k|i\,j| b_1|\cdots|b_{n-k-2})$.  
For each relation $i \independent j \mid K$ there are $|K|! \, (n-|K|-2)!$ such pairs.

A fan $\cF$ in $\R^n$ is said to be a {\em coarsening} of the $S_n$ fan if every cone in the $S_n$ fan is contained in a cone of $\cF$, or equivalently, if every cone of $\cF$ is a union of some cones of the $S_n$ fan. 
In particular, maximal cones of $\cF$ are unions of maximal cones of the $S_n$ fan, and $\cF$ can be constructed from the $S_n$ fan by removing certain walls (codimension one cones).  This gives an equivalence relation on $S_n$ --- two permutations are equivalent if and only if their corresponding cones in the $S_n$ fan are contained in the same cone in $\cF$.  Such an equivalence relation coming from a fan is called a {\em convex rank test} in~\cite{Morton_et_al}.  We will see in \S\ref{sec:causalInference} that for DAG models this equivalence relation coincides with that coming from the Sparsest Permutation Algorithm of Raskutti and Uhler~\cite{SP_alg}.

We identify a coarsening of the $S_n$ fan with the collection of walls which are removed.  Each wall corresponds to an adjacent pair of permutations as in (\ref{eqn:adjacent}), which gives a CI relation \mbox{$i \independent j \mid \{a_1,\dots,a_k\}$}.  It was shown in \cite[Theorem~6]{Morton_et_al} that a set of walls forms the missing walls in a fan that coarsens the $S_n$ fan if and only if the corresponding set of CI relations forms a semigraphoid.  In particular, if the wall associated to the pair (\ref{eqn:adjacent}) is not a wall in a coarsened $S_n$ fan $\cF$, then any pair obtained by permuting the $a$'s among themselves and the $b$'s among themselves is also not a wall in~$\cF$.

A complete fan $\cF$ in $\RR^n$ is called {\em polytopal} if it is the normal fan of a polytope. The $S_n$ fan itself is polytopal since it is the normal fan of a \emph{permutohedron} $P_n$ defined as follows.  Let $a_1 < a_2 < \dots < a_n$ be real numbers.  Let 
\[P_n = \conv\{(a_{\sigma(1)}, a_{\sigma(2)}, \dots, a_{\sigma(n)}) : \sigma \in S_n \} \subseteq \R^n.\]
Different choices of $a_i$'s give different polytopes but with the same normal fan.  We associate to each vertex of $P_n$ a permutation given by its descent vector as explained above, e.g.\ a point with coordinates $(2,3,4,1) \in \R^4$ is associated with its descent vector $(3|2|1|4)$, which is a permutation and {\em not} a point in $\RR^4$.
Two vertices of $P_n$ are connected by an edge if and only if their descent vectors differ by an adjacent transposition as in (\ref{eqn:adjacent}).  Thus each CI relation corresponds to a certain set of edges of~$P_n$.

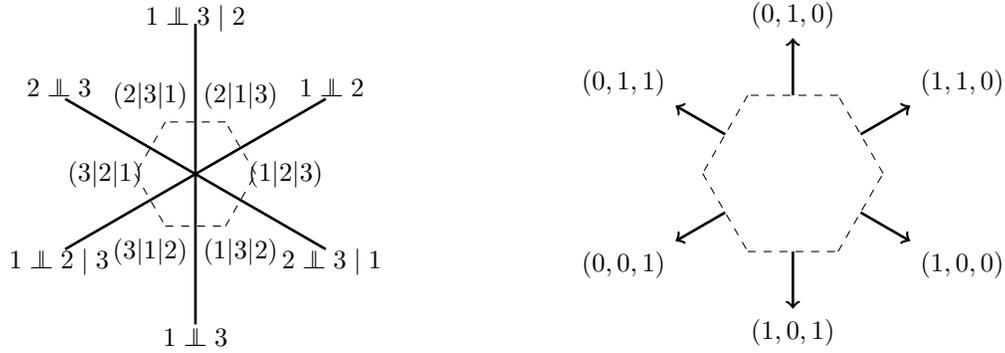
\begin{figure}[]
\begin{minipage}{0.45\textwidth}\begin{center}
\begin{tikzpicture}[scale=0.4,
                    every node/.style={font=\small}]
%normal fan
\draw[line width=1pt] (0,5) -- (0,-5);
\draw[line width=1pt] (4.33012701892,2.5) -- (-4.33012701892,-2.5);
\draw[line width=1pt] (-4.33012701892,2.5) -- (4.33012701892,-2.5);

\node at (0,5.2) {$1\independent 3\mid2$};
\node at (0,-5.4) {$1\independent 3$};
\node at (4.53012701892,2.9) {$1\independent 2$};
\node at (4.53012701892,-2.9) {$2\independent 3\mid 1$};
\node at (-4.53012701892,2.9) {$2\independent 3$};
\node at (-4.53012701892,-2.9) {$1\independent 2\mid 3$};

\node at (1.5, 2.59807621135) {$(2|1|3)$};
\node at (-1.5, 2.59807621135) {$(2|3|1)$};
\node at (3,0) {$(1|2|3)$};
\node at (-3,0) {$(3|2|1)$};
\node at (1.5, -2.59807621135) {$(1|3|2)$};
\node at (-1.5, -2.59807621135) {$(3|1|2)$};

%hexagon for reference
\draw[dashed] (1, 1.73205080757) -- (-1, 1.73205080757);
\draw[dashed] (-1, 1.73205080757) -- (-2,0);
\draw[dashed] (2,0) -- (1, 1.73205080757);
\draw[dashed] (-2,0) -- (-1, -1.73205080757);
\draw[dashed] (1, -1.73205080757) -- (2,0);
\draw[dashed] (-1, -1.73205080757) -- (1, -1.73205080757);
\end{tikzpicture}
\end{center}
(a) The $S_3$ fan modulo the line $(1,1,1)$. Maximal cones are labeled with permutations and the walls are labeled with CI relations.
\end{minipage}
\quad
\begin{minipage}{0.45\textwidth}\begin{center}
\begin{tikzpicture}[scale = 0.6,
                    every node/.style={font=\small}]
%hexagon
\draw[dashed] (1, 1.73205080757) -- (-1, 1.73205080757);
\draw[dashed] (-1, 1.73205080757) -- (-2,0);
\draw[dashed] (2,0) -- (1, 1.73205080757);
\draw[dashed] (-2,0) -- (-1, -1.73205080757);
\draw[dashed] (1, -1.73205080757) -- (2,0);
\draw[dashed] (-1, -1.73205080757) -- (1, -1.73205080757);

%outer normals, negate for inner normals
\draw[line width=1pt,->] (0,1.73205080757) -- (0,3);
\draw[line width=1pt,->] (0,-1.73205080757) -- (0,-3);
\draw[line width=1pt,->] (1.5,0.86602540378) -- (2.59807621135,1.5); 
\draw[line width=1pt,->] (-1.5,-0.86602540378) -- (-2.59807621135,-1.5);
\draw[line width=1pt,->] (-1.5,0.86602540378) -- (-2.59807621135,1.5);
\draw[line width=1pt,->] (1.5,-0.86602540378) -- (2.59807621135,-1.5);

\node at (0,3) [above] {$(0,1,0)$}; 
\node at (0,-3) [below] {$(1,0,1)$};
\node at (2.59807621135,1.5) [above right] {$(1,1,0)$};
\node at (2.59807621135,-1.5) [below right] {$(1,0,0)$};
\node at (-2.59807621135,1.5) [above left] {$(0,1,1)$};
\node at (-2.59807621135,-1.5) [below left] {$(0,0,1)$};
\end{tikzpicture}
\end{center}

(b) Permutohedron $P_3$ with outer normals of its facets.  \vspace{0.5cm}
\end{minipage}
\caption{\label{fig:S3P3} The permutohedron $P_3$ and its normal $S_3$-fan. Only the descent vectors, not coordinate vectors, of the vertices of $P_3$ are shown in (a). }
\end{figure}

A {\em generalized permutohedron} (see \cite{PostnikovReinerWilliams}) is a polytope whose normal fan is a coarsening of the $S_n$ fan. See Figures~\ref{fig:S3P3} and~\ref{fig:simplices} for some examples. These polytopes have other equivalent definitions, and are also called {\em $M$-convex polyhedra} or {\em base polyhedra}~\cite[(4.43)]{MurotaBook}.  Their projections along a coordinate direction give {\em generalized polymatroids}~\cite[Theorem~3.58]{FujishigeBook}.  We use the term ``generalized permutohedron'' to highlight the connection to permutations.

\begin{example}[Undirected graphical models and graph associahedra]
\label{sec_graphical_models}
%We now relate CI relations to undirected graphs: 
Let $G$ be an undirected graph with node set $[n]$. We associate a random variable $X_i$ to each node $i$ of the graph. The joint distribution $\mathbb{P}$ of the random vector $X=(X_1, \dots ,X_n)$ satisfies the \emph{undirected (global) Markov property} with respect to $G$ if $I\independent J\mid K$ for all disjoint subsets $I, J, K\subseteq [n]$ such that $K$ \emph{separates} $I$ and $J$ in $G$, i.e.\ every path between nodes $i\in I$ and $j\in J$ passes through a node $k\in K$. If a distribution $\mathbb{P}$ satisfies exactly the CI relations corresponding to separations in the graph $G$, then $\mathbb{P}$ is called \emph{faithful} or \emph{perfectly Markovian} with respect to~$G$. 

For any undirected graph there exist faithful regular Gaussian distributions; see~\cite[Chapter~3]{Lauritzen} for more details. Hence for any undirected graph $G$ the corresponding CI relations defined by the Markov property satisfy the gaussoid axioms. The coarsened $S_n$ fan associated to the gaussoid of an undirected graph is the normal fan of a polytope, which can be realized as the Minkowski sum of standard simplices \mbox{$\Delta_I = \conv\{e_i : i \in I\}$} where $I$ runs over all sets of nodes that induce connected subgraphs of $G$~\cite{Morton_et_al}.
These polytopes are called {\em graph associahedra} and were studied in~\cite{Devadoss, Carr_Devadoss, PostnikovReinerWilliams}.  
\qed 
\end{example}

\begin{figure}[t]
\begin{minipage}{0.25\textwidth}\begin{center}
\begin{tikzpicture}[scale=0.8]
\draw[line width=1pt] (0,0) -- (3,0);
\draw[line width=1pt] (0,0) -- (1.5,2.59807621135);
\draw[line width=1pt] (3,0) -- (1.5,2.59807621135);
\end{tikzpicture}
\end{center}
(a) The standard simplex $\conv\{e_1,e_2,e_3\}$. \vspace{1.5cm}
\end{minipage}
\quad
\begin{minipage}{0.25\textwidth}\begin{center}
\begin{tikzpicture}[scale=0.8]
\draw[line width=1pt] (0,0) -- (3,0);
\draw[line width=1pt] (0,0) -- (1.5,-2.59807621135);
\draw[line width=1pt] (3,0) -- (1.5,-2.59807621135);
\end{tikzpicture}
\end{center}
(b) The matroid polytope of $U_{2,3}$, $\conv\{e_1+e_2, e_1+e_3, e_2+e_3\}$. \vspace{1cm}
\end{minipage}
\quad
\begin{minipage}{0.33\textwidth}\begin{center}
\begin{tikzpicture}[scale=0.3,
                    every node/.style={font=\small}]
%pentagon
\draw[line width=1pt] (6, 3.46410161513775) -- (-2, 3.46410161513775);
\draw[line width=1pt] (-2, 3.46410161513775) -- (-4,0);
\draw[line width=1pt] (4,0) -- (6, 3.46410161513775);
\draw[line width=1pt] (-4,0) -- (-2, -3.46410161513775);
\draw[line width=1pt] (2, -3.46410161513775) -- (4,0);
\draw[line width=1pt] (-2, -3.46410161513775) -- (2, -3.46410161513775);
\end{tikzpicture}
\end{center}
(c) The DAG associahedron of $1 \rightarrow 3 \leftarrow 2$ with CI relation $1 \independent 2$.  This polytope is not a Minkowski sum of scaled standard simplices.
\end{minipage}
\caption{\label{fig:simplices} Some generalized permutohedra. Compare with the fan in Figure~\ref{fig:S3P3}.}
\end{figure}
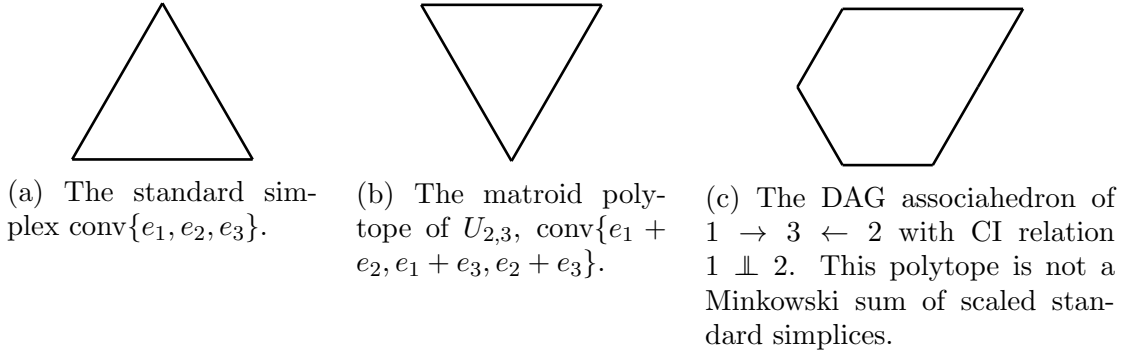

We will now summarize a characterization of coarsened $S_n$ fans that are polytopal, based on~\cite{Morton_et_al} and~\cite{HMSSW}.   
Let $2^{[n]}$ denote the power set of $[n]$, the set of all subsets of $[n]$. 
A function $\omega : 2^{[n]} \rightarrow \RR$ is called {\em submodular} if
\begin{equation}
\label{eqn:submodular}
\omega(Ki) + \omega(Kj) \geq \omega(Kij)+ \omega(K) 
\end{equation}
 for all $K\subseteq [n]$ and $i,j \in[n] \backslash K$.
A submodular function also satisfies $\omega(A) + \omega(B) \geq \omega(A\cup B)+ \omega(A \cap B)$ for all $A,B \subseteq [n]$.  Note that a submodular function on $2^{[n]}$ is an {\em $L$-convex} function on the unit cube $\{0,1\}^n$~\cite{MurotaBook}. 

\begin{defn}
\label{def:submodular}
 A semigraphoid on $[n]$ is called {\em submodular} if there is a submodular function $\omega$ on $2^{[n]}$
 with $\omega(\emptyset) = 0$
such that $\omega(Ki) + \omega(Kj) = \omega(Kij)+ \omega(K)$  if and only if the relation $i \independent j \mid K$ is in the semigraphoid.
\end{defn}
Submodular semigraphoids correspond to {\em structural independence models}~\cite[\S 5.4.2]{Studeny}, which can be viewed as semigraphoids obtained from supermodular functions, whose negatives are submodular functions. 

The following result shows that every submodular function determines a semigraphoid, and the semigraphoids that arise this way are precisely those corresponding to polytopal coarsenings of the $S_n$ fan.

\begin{lem}
\label{lem:submodular}
A polytope $P \subseteq \RR^n$ is a generalized permutohedron if and only if there exists a submodular function $\omega : 2^{[n]} \rightarrow \RR$ with $\omega(\varnothing) = 0$ such that 
\begin{equation} 
\label{eqn:polytope}
P = \{x \in \RR^n : \sum_{i \in I} x_i \leq \omega(I) \text{ for each nonempty } I \subseteq [n], \text{ and } 
\sum_{i\in[n]}x_i = \omega([n])\}.
\end{equation}
A wall in the $S_n$ fan corresponding to $i \independent j \mid K$ is missing in the normal fan of $P$ defined by $\omega$ as above if and only if $\omega(Ki) + \omega(Kj) = \omega(Kij)+ \omega(K)$.
In particular, a coarsened $S_n$ fan is polytopal if and only if the corresponding semigraphoid is submodular.  
\end{lem}

The lemma follows from the conjugacy between $L$- and $M$- convex functions and also from~\cite[Theorem~4.15]{MurotaBook}.
%which was ``well known, but neither precise statement nor proof can be found in the literature''~\cite[Chapter~4 Bibliographical Notes]{MurotaBook}.  
A part of it appeared in \cite[Proposition~12 and Theorem~14]{Morton_et_al}.  We provide a proof in Appendix~\ref{sec:submodular}, as it is difficult to find a complete proof in the literature.

\begin{remark}
\label{rem:flip}
If $\omega$ is a submodular function on $2^{[n]}$ with $\omega(\emptyset) = 0$, then $\omega' : 2^{[n]} \rightarrow \RR$ defined as $\omega'(S) = \omega([n]\backslash S) - \omega([n])$ is also submodular with $\omega'(\emptyset)=0$. The polytopes $P$ and $P'$, defined by $\omega$ and $\omega'$ as in~(\ref{eqn:polytope}), are related by $-P = P'$.
\qed
\end{remark}

\begin{example}
\label{ex:simplex}
Consider the submodular function $\omega$ on $2^{[n]}$ whose value is~$1$~on all nonempty sets and~$0$~on the empty set.  This function is known as the rank function of the uniform rank one matriod on $[n]$.  The generalized permutohedron defined by this submodular function is a standard simplex of dimension $n-1$ whose outer normal vectors are $e_I$ for subsets $I$ of size $n-1$.  Any set of $n-2$ facet normals spans a wall in the normal fan, with pairs of the form~(\ref{eqn:adjacent}), where $K = \emptyset$, corresponding to relations of the form $i \notindependent j \mid \emptyset$.  See Figures~\ref{fig:S3P3} and~\ref{fig:simplices}. %{\color{red} It seems to me that rev. 3 suggests $\notindependent$ to indicate that the semigraphoid does not contain these relations, rather than stating what the walls correspond to. I'm not sure which to use.}
\qed
\end{example}

This characterization leads to the following questions for any given semigraphoid:

\vspace{0.2cm}
\paragraph{\bf Question A} \emph{Is a given semigraphoid submodular? And if so, can we construct a submodular function with the desired equalities as in Definition~\ref{def:submodular}?}
\vspace{0.2cm}

In the following sections we will give a positive answer to these questions for semigraphoids coming from DAG models.

Rank functions of matroids are submodular functions, so every matroid $M$ on the ground set $[n]$ gives a semigraphoid on $[n]$ as follows:
\begin{align*}
%\label{eqn:matroid}
\begin{split}
i \notindependent j \mid K & \iff \rank(Ki) + \rank(Kj) > \rank(Kij)+ \rank(K)
% &  \iff \rank(K) + 1 = \rank(Ki) = \rank(Kj) = \rank(Kij).
\end{split}
\end{align*}
%The first equivalence follows from Lemma~\ref{lem:submodular} above
Note that since a matroid rank function takes integer values and $\rank(Aa) \leq \rank(A)+1$ for any $A \subseteq [n]$ and $a \in [n]$, we obtain
\begin{align}
\label{eqn:matroid}
\begin{split}
\qquad i \notindependent j \mid K &  \iff \rank(K) + 1 = \rank(Ki) = \rank(Kj) = \rank(Kij).
\end{split}
\end{align}

In this case, the coarsening of the $S_n$ fan is the outer normal fan of the \emph{matroid polytope}, which is defined as the convex hull of the indicator functions of the bases of the matroid.  For example, the standard simplex $\Delta_I = \conv\{e_i : i \in I\}$ is the matroid polytope of the rank one matroid in which each element of $I$ forms a base.  The intersection of two semigraphoids (as sets of conditional independence relations) is again a semigraphoid.  This intersection operation corresponds to common refinement, Minkowski sum, and sum, respectively, for fans, polytopes, and submodular functions.

\vspace{0.2cm}

\paragraph{\bf Question B} \emph{Which submodular semigraphoids can be obtained from the sums of rank functions of matroids (as in Definition~\ref{def:submodular})? 
  Which fans arising from semigraphoids are normal fans of Minkowski sums of matroid polytopes?}

\vspace{0.2cm}

For example, the Minkowski sum of all standard simplices $\Delta_I$, for all nonempty subsets $I \subseteq [n]$, is affinely equivalent to the permutohedron $P_n$, i.e.\ they have the same normal fan, which is the entire $S_n$ fan.  This decomposition is not unique, however, e.g.\ $P_3$ is a hexagon and can be decomposed as the Minkowski sum of either two triangles or three line segments, all of which are matroid polytopes.

\begin{example}
\label{ex:notMSS}
Let $\mathcal{G}$ be the following DAG.
\begin{center}
	\begin{tikzpicture}
		\tikzset{vertex/.style = {shape=circle,draw,minimum size=1.5em}}
	\tikzset{edge/.style = {->,> = latex'}}
		\tikzset{edge2/.style = {- = latex'}}
	\node[draw=none,fill=none] (1) at (-1,1) {1};
	\node[draw=none,fill=none] (3) at (0,0) {3};
	\node[draw=none,fill=none] (2) at (1,1) {2};
	\draw[edge] (1) to (3);
	\draw[edge] (2) to (3);
	\end{tikzpicture}
\end{center}
We will see in the next section that the Markov property on $\mathcal{G}$ defines a single CI relation, namely $1 \independent 2$.  Removing the corresponding wall in the $S_3$ fan gives a fan with 5 maximal cones.  Figure~\ref{fig:simplices}(c) depicts a polytope with this normal fan.  It is straightforward to check that this fan is not the normal fan of a Minkowski sum of standard simplices, but it is the normal fan of the Minkowski sum of the simplex in Figure~\ref{fig:simplices}(b) together with two line segments, which are all matroid polytopes.
\qed
\end{example}

%% file: 3_BN.tex
\section{Bayesian networks}

Similarly to undirected graphs we can define probabilistic models on DAGs. Such graphical models are also known as \emph{Bayesian networks}. 

Let $\mathcal{G}$ be a DAG with nodes $[n]$. If there is a directed edge from $i$ to $j$ in $\mathcal{G}$, which we denote by $i\rightarrow j$ in $\mathcal{G}$ or $(i,j) \in \mathcal{G}$, the node $i$ is called a \emph{parent} of the node $j$.  The set of all parent nodes of $j$ is denoted by $\emph{\textrm{pa}(j)}$. 

We now review the concept of separation for DAGs. A {\em path} in $\mathcal{G}$ is an alternating sequence of nodes and edges, starting and ending at nodes, in which each edge is adjacent in the sequence to its two endpoints\footnote{This is often called a ``walk'', but we prefer to use ``path'' in order to be consistent with the notion of a ``Bayes ball path''.}. The path may contain repeated edges and nodes.  We do {\em not} assume that the direction of the edges is compatible with the ordering of the nodes in the path.  

\begin{defn}
Let $\mathcal{G}$ be a DAG on $[n]$ and let $i,j \in [n]$ and $K \subseteq [n]\setminus\{i,j\}$.  A \emph{Bayes ball path from $i$ to $j$ given $K$} in $\mathcal{G}$ is a path from $i$ to $j$ in $\mathcal{G}$ such that 
\begin{enumerate}
\item if $a \rightarrow b \rightarrow c$ or $a \leftarrow b \rightarrow c$ or $a \leftarrow b \leftarrow c$ is on the path, then $b \notin K$;
\item if $a \rightarrow b \leftarrow c$ is on the path, then $b \in K$ (where $a$ and $c$ need not be distinct). In this case the node $b$ is called a {\em collider} along the path.
\end{enumerate}

\end{defn}

See Figures~\ref{graphs_ex_1} and~\ref{fig:8nodes} for examples. Informally we think of a directed edge $i \rightarrow j$ as pointing {\em down} from $i$ to $j$.  A ``Bayes ball'' rolls along edges of the DAG.  It cannot roll through nodes that are in $K$, but it can ``bounce off'' them by going down, touching $K$, then going back up  along either the same or a different edge.

%{\color{red} rev. 3 suggests directionally separated instead of directed separated.}
For subsets of nodes $I,J,K\subseteq [n]$, we say that $I$ and $J$ are \emph{directionally separated} or \emph{d-separated} by $K$ in $\mathcal{G}$ %, denoted by $A \perp_\mathcal{G} B\mid C$,
if there is no Bayes-ball path from any element of $I$ to any element of $J$ given $K$~\cite{Verma_Pearl}. This led to the construction of the \emph{Bayes-Ball algorithm}~\cite{Shachter}, an algorithm for determining d-separation statements. 
Similarly as for undirected graphs, we can also associate a random vector with joint distribution $\mathbb{P}$ to the nodes of a DAG $\mathcal{G}$. Then $\mathbb{P}$ satisfies the \emph{directed (global) Markov property} with respect to $\mathcal{G}$ if $I\independent J\mid K$ for all disjoint subsets $I, J, K\subset V$ such that % $A \perp_\mathcal{G} B\mid C$
$K$ d-separates $I$ and $J$ in $\mathcal{G}$. 
A faithful distribution to $\mathcal{G}$, i.e.~a distribution that satisfies exactly the CI relations corresponding to d-separation in $\mathcal{G}$, can be realized by regular Gaussian distributions (see \S\ref{sec:multiinformation}). Hence, for any DAG $\mathcal{G}$ the CI relations of the form $i \ind j \mid K$, where $i$ and $j$ are d-separated given $K$ in $\mathcal{G}$, form a gaussoid, which we call a \emph{DAG gaussoid}.

It is important to note that while the set of separation statements uniquely determines an undirected graph, this is not the case for d-separation statements for DAGs. Two DAGs are called \emph{Markov equivalent} if they imply the same d-separation statements. The Markov equivalence class is determined by the skeleton of a DAG and its \emph{V-structures} --- triples of nodes $(i,j,k)$ such that $i\to k \leftarrow j$ and $i,j$ are not adjacent~\cite{DAG_equivalence}. An \emph{essential graph}~\cite{DAG_equivalence} (also called a {\em completed partially directed acyclic graph} or \emph{CPDAG} in~\cite{Chickering} and a \emph{maximally oriented graph} in~\cite{Meek}) is a graph with undirected and directed edges that uniquely represents a Markov equivalence class of DAGs. It has the same skeleton as the DAGs in the Markov equivalence class and contains a directed edge $i\to j$ if and only if each DAG in the Markov equivalence class contains the directed edge $i\to j$. 

The following is our main result and answers Questions A and B for DAG gaussoids.
\begin{thm}[Main Theorem]
\label{thm:main}
Every DAG gaussoid is submodular.  Equivalently, the associated coarsening of the $S_n$ fan is the normal fan of a polytope.  Moreover, there is a realization of this polytope as a Minkowski sum of matroid polytopes.
\end{thm}

The equivalence of the first two statements follows from Lemma~\ref{lem:submodular} above. We call any such polytope resulting from a DAG gaussoid a {\em DAG associahedron}. DAG associahedra are uniquely defined up to equivalence of normal fans, and they only depend on the DAG up to Markov equivalence.  

\begin{remark}
Let $\mathcal{G}$ be a DAG. The normal fan of the DAG associahedron corresponding to $\mathcal{G}$ can be obtained by coarsening the normal fan of the graph associahedron corresponding to the \emph{moral graph} of $\mathcal{G}$ --- the undirected graph with edges $(i,j)$ if $i\to j$ in $\mathcal{G}$, $j\to i$ in $\mathcal{G}$, or $i\to k\leftarrow j$ for some $k$ in $\mathcal{G}$; see Figure~\ref{graphs_ex_1}. \qed
\end{remark}

In the next two sections, we will give two independent proofs for the submodularity of DAG gaussoids.  
In the first proof, in \S\ref{sec:multiinformation}, we use  multiinformation, or relative entropy, to give a formula for the submodular function and hence a realization of DAG associahedra. However, in general the constant terms of the inequalities in this construction are not rational.  We will discuss some heuristic methods for finding exact combinatorial information from approximate inequalities.
In the second proof, in \S\ref{sec:MSMP}, we give a realization of DAG associahedra as Minkowski sums of matroid polytopes, which are integral polytopes.  The submodularity of a semigraphoid can be tested using linear programming~\cite{HMSSW}. So our theorem states that the linear programs coming from DAG gaussoids are always feasible, and our proofs give an explicit construction of a feasible solution.

We illustrate the concepts introduced so far with an example of a DAG model on 4 nodes and describe the corresponding DAG associahedron. 

\begin{example}
\label{ex:4nodes}

Consider the DAG $\mathcal{G}$ shown in Figure~\ref{graphs_ex_1}. An example of a Bayes ball path in $\mathcal{G}$ is the path from node 1 to 2 given $K=\{4\}$, since on the path $1\to 3 \to 4 \leftarrow 3 \leftarrow 2$ the node $3\notin K$ but $4\in K$. 
%One can easily check that the d-separation statements for $\mathcal{G}$ are:
%$$1\perp_\mathcal{G} 2, \quad 1\perp_\mathcal{G} 4\mid 3, \quad 2\perp_\mathcal{G} 4\mid 3,  \quad \{1,2\}\perp_\mathcal{G} 4\mid 3, \quad 1\perp_\mathcal{G} 4\mid \{2, 3\}, \quad 2\perp_\mathcal{G} 4\mid \{1, 3\}.$$
%Since the d-separation statement $A\perp_\mathcal{G} B\mid C$ is equivalent to $a\perp_\mathcal{G} b\mid C$ for all $a\in A$ and $b\in B$, we concentrate on statements about singletons, i.e.~where $|A|=|B|=1$.
The DAG gaussoid corresponding to $\mathcal{G}$ consists of the CI relations
$$1\independent 2, \quad 1\independent 4\mid 3, \quad 2\independent 4\mid 3, \quad 1\independent 4\mid \{2, 3\}, \quad 2\independent 4\mid \{1, 3\}.$$
The corresponding edges of the permutohedron are shown in green and blue in Figure~\ref{perm_ex_1}(a).  Since these CI relations form a semigraphoid, we obtain a coarsening of the $S_n$ fan by removing the edges $(1 2 |3|4)$, $(1 2 |4|3)$, $(3 |1 4|2)$, $(3 |2 4|1)$, $(2 |3|1 4)$, $( 3|2|1 4)$, $(1 |3|2 4)$ and $(3 |1|2 4)$. The resulting coarsening of the $S_n$ fan obtained by contracting the colored edges in the permutohedron is polytopal. The convex polytope corresponding to this DAG associahedron is shown in Figure~\ref{perm_ex_1}(c). 

\begin{figure}[!t]
	\centering
	\begin{tikzpicture}
		\tikzset{vertex/.style = {shape=circle,draw,minimum size=1.5em}}
	\tikzset{edge/.style = {->,> = latex'}}
		\tikzset{edge2/.style = {- = latex'}}

	\node[draw=none,fill=none] (1) at (3,0) {1};
	\node[draw=none,fill=none] (2) at (5,0) {2};
	\node[draw=none,fill=none] (3) at (4,-1) {3};
	\node[draw=none,fill=none] (4) at (4,-2) {4};

	\draw[edge] (1) to (3);
	\draw[edge] (2) to (3);
	\draw[edge] (3) to (4);

	\node[draw=none,fill=none] (1) at (9,0) {1};
	\node[draw=none,fill=none] (2) at (11,0) {2};
	\node[draw=none,fill=none] (3) at (10,-1) {3};
	\node[draw=none,fill=none] (4) at (10,-2) {4};

	\draw[edge2] (1) to (3);
	\draw[edge2] (1) to (2);
	\draw[edge2] (2) to (3);
	\draw[edge2] (3) to (4);
	
	\end{tikzpicture}

\caption{\label{graphs_ex_1}The DAG $\mathcal{G}$ (left) and its moral graph $G$ (right) discussed in Example~\ref{ex:4nodes}.}
\end{figure}
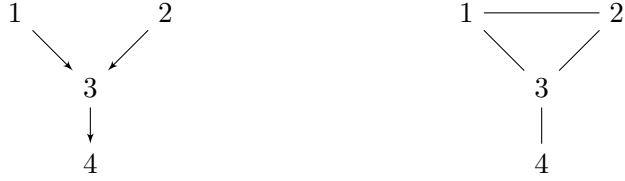

\begin{figure}[!b]
\centering
\begin{minipage}{0.3\textwidth}
\begin{center}
\includegraphics[scale=0.44]{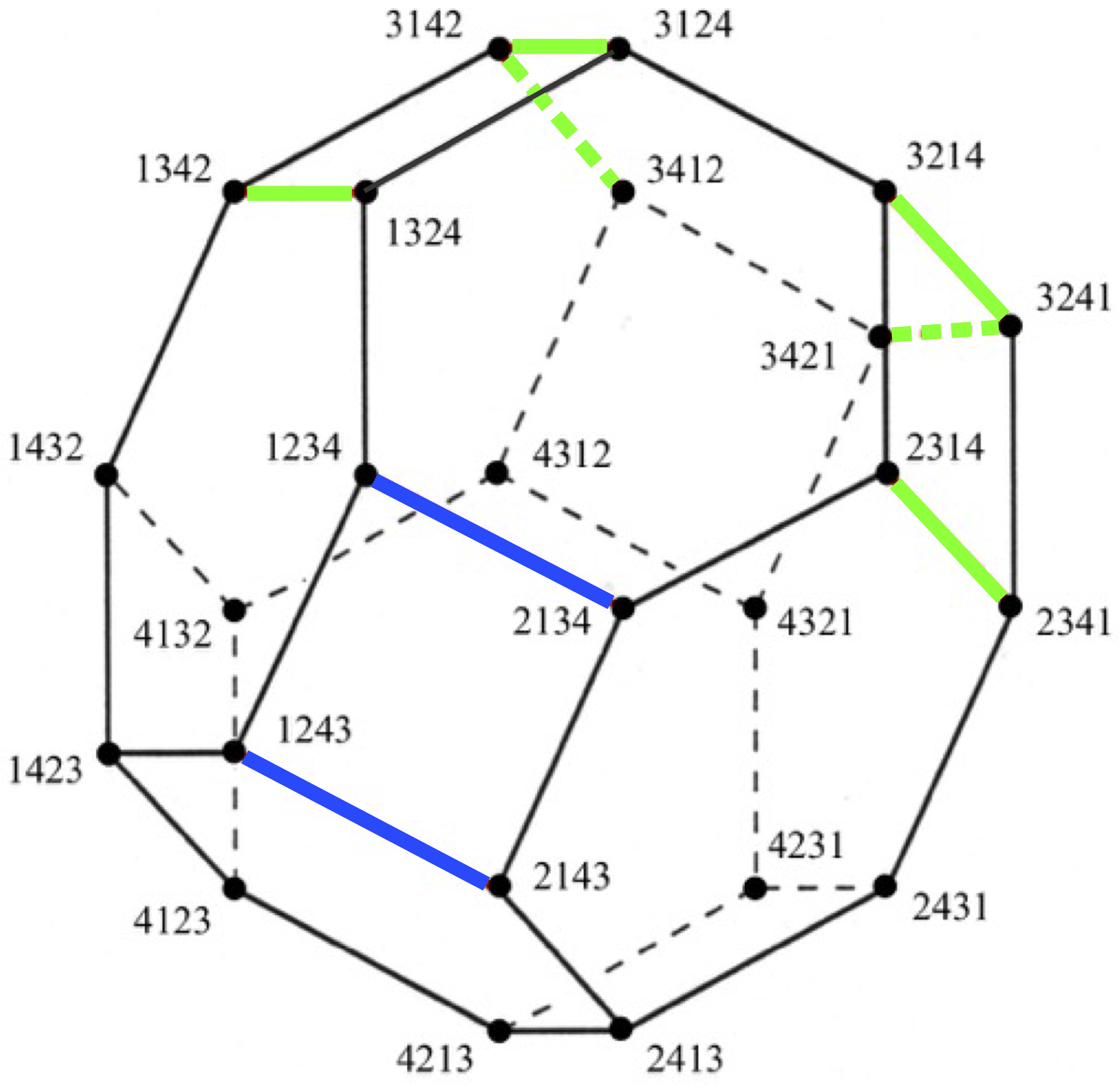}\label{perm_1}
\end{center}
(a) Permutohedron $P_4$. The green edges correspond to CI relations in the moral graph $G$ in Example~\ref{ex:4nodes}.  The blue edges correspond to the additional CI relations in $\mathcal{G}$.
\end{minipage}
\;\;
\begin{minipage}{0.3\textwidth}
\begin{center}
\includegraphics[scale=0.33]{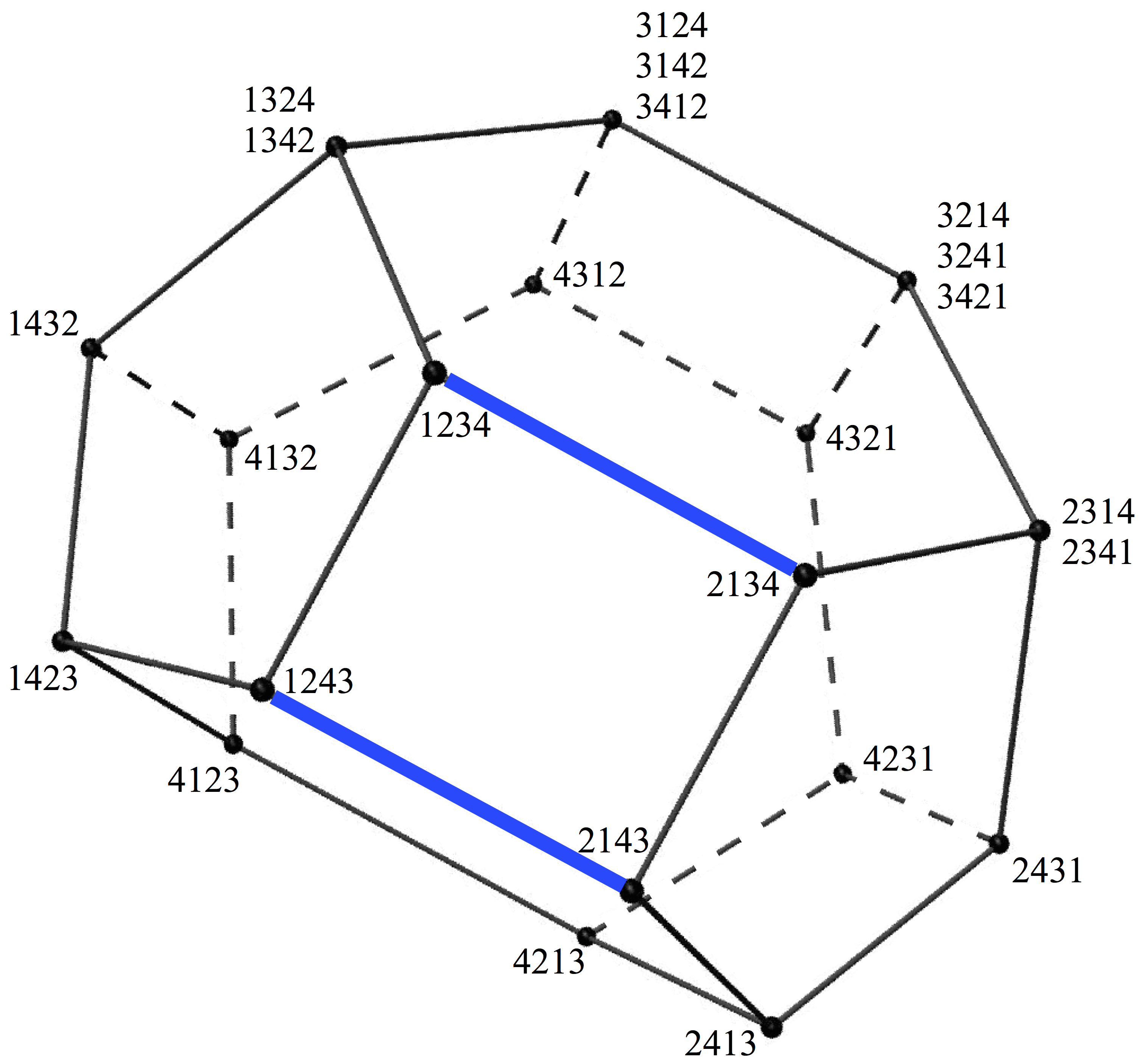}\label{moral_ass}
\end{center}
(b) The graph associahedron of the moral graph $G$ obtained by contracting the green edges. \vspace{0.8cm}
\end{minipage}
\;\;
\begin{minipage}{0.3\textwidth}
\begin{center}
\includegraphics[scale=0.35]{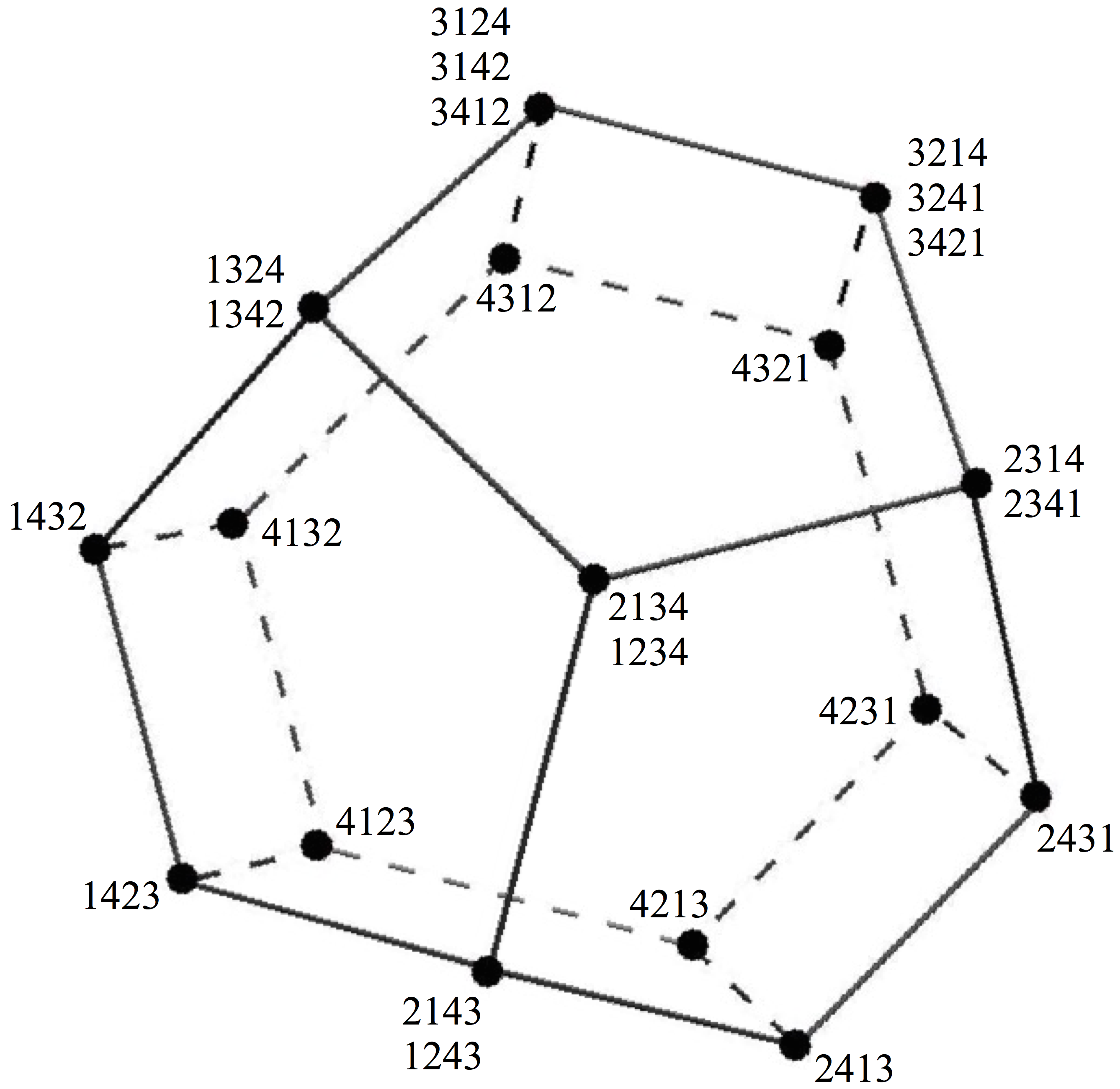}\label{associahedra_1}
\end{center}
(c) DAG associahedron for $\mathcal{G}$ obtained by contracting both, green and blue edges. \vspace{0.8cm}
\end{minipage}
\caption{\label{perm_ex_1}  The vertices are labeled by descent vectors of permutations, with ``$\mid$''s removed.  The figures show how the combinatorics of the polytope changes as edges are contracted, but the polytopes are not drawn to be geometrically correct.}  
\end{figure}

The moral graph $G$ of $\mathcal{G}$ is shown in Figure~\ref{graphs_ex_1} (right). The gaussoid corresponding to $G$ consists of the CI relations
$$1\independent 4\mid 3, \quad 2\independent 4\mid 3, \quad 1\independent 4\mid \{2, 3\}, \quad 2\independent 4\mid \{1, 3\}.$$
In general, any DAG gaussoid contains the gaussoid of its moral graph. The edges corresponding to the CI relations for the moral graph are shown in green in Figure~\ref{perm_ex_1}(a). By contracting the green edges in the permutohedron we obtain the graph associahedron corresponding to $G$ shown in Figure~\ref{perm_ex_1}(b). By further contracting also the blue edges, we obtain the DAG associahedron corresponding to $\mathcal{G}$.

%are shown in different colors in Figure~\ref{?}. The corresponding graph associahedron is attained by contracting the colored edges. This results in a new convex polytope  shown in Figure~\ref{?}.

As we will see in Proposition~\ref{prop:MSS-Vstructs}, the DAG associahedron in this example cannot be realized as a Minkowski sum of simplices.  However, we will show in \S\ref{sec:MSMP} that it can be realized as the following Minkowski sum of matroid polytopes:
\[\Delta_{13} + \Delta_{23} + \Delta_{34} + \Delta_{134} + \Delta_{234} + \conv\{e_{12},e_{13},e_{23}\} + \conv\{e_{12},e_{13},e_{23},e_{14},e_{24}\}.\]
As we will see in \S\ref{sec:MSMP}, the first three polytopes in the sum correspond to the three edges in the DAG, the next two correspond to the paths $1\,3\,4$ and $2\,3\,4$, which have no colliders, and the last two correspond to the paths $1\,3\,2$ (given $3$) and $1\,3\,4\,3\,2$ (given $4$) respectively.  See Example~\ref{ex:4nodesAgain}.
\qed
\end{example}

We end this section with two observations about DAG associahedra. In the following example we show that unlike graph associahedra, DAG associahedra need not be \emph{simple}.%,  as the following example shows.
\begin{example}[A non-simple DAG associahedron]
\label{ex:notSimple}

Let $\mathcal{G}$ be the following DAG.
\begin{center}
	\begin{tikzpicture}
		\tikzset{vertex/.style = {shape=circle,draw,minimum size=1.5em}}
	\tikzset{edge/.style = {->,> = latex'}}
		\tikzset{edge2/.style = {- = latex'}}
	\node[draw=none,fill=none] (1) at (0,1) {1};
	\node[draw=none,fill=none] (2) at (1,1) {2};
	\node[draw=none,fill=none] (3) at (1,0) {3};
	\node[draw=none,fill=none] (4) at (0,0) {4};
	\draw[edge] (1) to (2);
	\draw[edge] (2) to (3);
	\draw[edge] (3) to (4);
	\draw[edge] (1) to (4);
	\draw[edge] (2) to (4);
	\end{tikzpicture}
\end{center}
 The corresponding DAG gaussoid consists only of the CI relation $1 \independent 3 \mid 2$, which corresponds to a single edge $2|13|4$ on the permutohedron $P_4$.  Contracting this edge gives a vertex adjacent to 4 edges on a 3-dimensional polyhedron, so the resulting polytope is not simple.  In this case, the combinatorial operation of contracting the edge can be realized geometrically by pushing the two neighboring square facets toward each other until they meet at a vertex. In other words, the edge shared by two hexagonal faces contracts to a single vertex.
\qed
\end{example}

Furthermore, as already mentioned in Example~\ref{ex:4nodes}, unlike graph associahedra, DAG associahedra need not be Minkowski sums of standard simplices (MSS) in the sense of \cite{Morton_et_al}. In fact, the following result shows that a DAG associahedron can be realized as a MSS if and only if the DAG gaussoid equals the gaussoid of its moral graph, or in other words, if and only if the DAG model coincides with an undirected graphical model.

\begin{prop}
\label{prop:MSS-Vstructs}
The DAG associahedron associated to a DAG $\mathcal{G}$ is MSS if and only if $\mathcal{G}$ does not contain any V-structures, i.e.\ the DAG model coincides with an undirected graphical model.
\end{prop}

We saw in Example~\ref{ex:notMSS} that a V-structure cannot be MSS.  We can generalize this example to the following consequence of~\cite[Proposition~20]{Morton_et_al}:

\begin{lem}
\label{lem:MSS}
If a semigraphoid arises from a Minkowski sum of standard simplices, then for any $i,j\in[n]$ distinct and $K \subseteq K' \subseteq [n]\setminus\{i,j\}$, we have
\[
i \independent j \mid K  \implies i \independent j \mid K' .
\] 
\end{lem}

\begin{proof}
For $I \subseteq [n]$, the standard simplex $\Delta_I$ is the matroid polytope of the rank one matroid on $[n]$ whose loops are $[n] \backslash I$.  By~(\ref{eqn:matroid}) the semigraphoid corresponding to $\Delta_I$ contains $i \notindependent j \mid K$ where $i,j \in I$ and $K \cap I = \emptyset$.  Taking a Minkowski sum of such simplices corresponds to taking the union of the associated conditional dependence statements.
It follows that $i \notindependent j \mid K' \implies i \notindependent j \mid K$ for all $K \subseteq K' \subseteq [n]\setminus\{i,j\}$.
\end{proof}

Using Lemma~\ref{lem:MSS} we can now easily prove Proposition~\ref{prop:MSS-Vstructs}.

\begin{proof}[Proof of Proposition~\ref{prop:MSS-Vstructs}]
If $\mathcal{G}$ does not contain any V-structures, then the corresponding DAG gaussoid  is equivalent to the gaussoid obtained from an undirected graph, namely the skeleton of $\mathcal{G}$, so it is MSS.  

Conversely, suppose that $\mathcal{G}$ contains a V-structure $i \rightarrow \ell \leftarrow j$.  Let $K\subset [n]$ be the set of nondescendants of $i$ and $j$ in $\cG$, i.e.\ the set of $k \in [n]\setminus\{i,j\}$ such that there is no directed path from $i$ to $k$ or from $j$ to $k$ in $\cG$.  Then the CI relation $i \independent j \mid K$ is contained in the gaussoid corresponding to $\mathcal{G}$. However, the CI relation $i \independent j \mid K \ell$ is not in the gaussoid of $\mathcal{G}$, since there is a Bayes ball path from $i$ to $j$ given $K\ell$ in $\mathcal{G}$. Hence by Lemma~\ref{lem:MSS} above, the DAG associahedron corresponding to $\mathcal{G}$ is not MSS.\end{proof}

%% file: 4_multiinformation.tex
\section{A construction of DAG associahedra from multiinformation}
\label{sec:multiinformation}

The {\em multiinformation} of a probability measure $\mathbb{P}$ on $[n]$ is a function $m_\mathbb{P} : 2^{[n]} \rightarrow [0, \infty]$ defined by
$$
m_\mathbb{P}(S) = H(\mathbb{P} | \Pi_{i\in S}\mathbb{P}^{\{i\}}),
$$
where $H$ denotes the relative entropy with respect to a product of one-dimensional marginals~$\mathbb{P}^{\{i\}}$.  For the case of most interest to us, when $\mathbb{P}$ is a regular Gaussian, there is a simpler formula as follows. Let $\mathbb{P}$ be a regular Gaussian measure on $[n]$ with covariance matrix $\Sigma$.  Let $\Gamma$ be the correlation matrix of $\mathbb{P}$ --- a symmetric positive definite matrix obtained from $\Sigma$ by simultaneously rescaling the rows and columns so that all the diagonal entries are equal to one.  In other words, $\Gamma = D^{-1/2}\Sigma D^{-1/2}$ where $D = \textrm{diag}(\Sigma)$. Then we have 
\begin{equation}
\label{eqn:GaussianCIs}
i \independent j \mid K \iff \rank(\Gamma_{Ki,Kj}) \leq |K|
\end{equation}
where $\Gamma_{A,B}$ denotes the submatrix of $\Gamma$ with rows and columns indexed by $A$ and $B$, respectively~\cite{Sullivant}.  
By~\cite[Corollary~2.6]{Studeny} the multiinformation $m_\mathbb{P}(A)$ for $A \subseteq [n]$ is 
\[
m_\mathbb{P}(A) = -\frac{1}{2}\log \det (\Gamma_{A,A}).
\]
Since $\Gamma$ is positive definite, all its principal minors $\det(\Gamma_{A,A})$ are nonzero.  We define $\det (\Gamma_{\emptyset,\emptyset})$ to be~$1$. By~\cite[Corollary~2.2]{Studeny} we have
\[ m_\mathbb{P}(A) = 0 \text{ for all }A \subseteq [n], |A| \leq 1, \text{ and }\]
\[m_\mathbb{P}(ABC) + m_\mathbb{P}(C) \geq m_\mathbb{P}(AC) + m_\mathbb{P}(BC) \text{ for all }A,B,C \subset [n]\]
with equality if and only if $A \independent B \mid C$ under $\mathbb{P}$. 

We summarize this discussion in the following lemma.
\begin{lem}
\label{lem_Gaussian_multiinf}
If $\mathbb{P}$ is a regular Gaussian distribution with correlation matrix $\Gamma$, then its semigraphoid is submodular, with submodular function given by
\[
%A \mapsto \frac{1}{2} \log\det (\Gamma_{A,A}).
A \mapsto \log\det (\Gamma_{A,A}).
\] %{\color{red} rev. 3 suggests removing the 1/2 for clarity here, since scaling does not affect submodularity.}
\end{lem}

The submodularity of DAG gaussoids (i.e.~the first part of Theorem~\ref{thm:main}) follows from the lemma above and the fact that any DAG gaussoid has a faithful regular Gaussian realization. See for example~\cite[\S3.3]{DrtonSturmfelsSullivant}, where the following construction is described.

Let $\mathcal{G}$ be a DAG on the nodes $[n]$. Assume that the nodes are labeled so that if $i \rightarrow j$ is an edge in $\mathcal{G}$, then $i < j$.  Let $\Lambda$ be an upper-triangular matrix whose entries have the form
\[
\Lambda_{i,j} = \left\{ 
\begin{array}{ll} 
1 & \text{ if } i = j,  \\
-\ell_{ij} & \text{ if } i \rightarrow j \text{ is an edge in }\mathcal{G}, \\
0 & \text{ otherwise, }
\end{array} \right. 
\]
where $\ell_{ij}$ are real numbers.  Let $K=\Lambda \Lambda^T$ and $\Sigma = K^{-1}$.  Then $K$ is symmetric positive definite by construction, and so is $\Sigma$.  For almost all choices of real numbers $\ell_{ij}$ (apart from an algebraic hypersurface), a Gaussian distribution $\mathbb{P}$ with covariance matrix $\Sigma$ is faithful to the DAG gaussoid of $\mathcal{G}$~\cite{URBY}.

In fact, as explained in the following lemma, the inequalities for the desired generalized permutohedron can also be computed directly from minors of $K = \Lambda \Lambda^T$ instead of from the correlation matrix $\Gamma=D^{-1/2}\Lambda^{-T}\Lambda^{-1}D^{-1/2}$, where $D=\textrm{diag}(\Lambda^{-T}\Lambda^{-1})$. This result simplifies computations considerably since we do not need to perform any matrix inversion on $\Lambda\Lambda^T$.

\begin{lem}
\label{lem:multiinformation}
Let $K$ be a positive definite matrix and let $\omega$ be the submodular function on $2^{[n]}$ given by $\omega(A) = \log \det (K_{A,A})$. Let $P$ be the polytope defined as in~(\ref{eqn:polytope}).  Then $-P$ is the generalized permutohedron corresponding to the semigraphoid of a regular Gaussian distribution $\mathbb{P}$ with covariance matrix $\Sigma=K^{-1}$.
\end{lem}

\begin{proof}
 The polytope defined by the submodular function $A \mapsto \log \det(\Gamma_{A,A})$ is obtained from the polytope defined by the submodular function $A \mapsto \log \det(\Sigma_{A,A})$ by translation in each coordinate direction~$i$ by $-\log \Sigma_{i,i}$.  Thus these two polytopes have the same normal fans and encode the same semigraphoids.

 For $A \subseteq [n]$ and $B = [n] \backslash A$, we have $(\Sigma_{A,A})^{-1} = K_{A,A} - K_{A,B}(K_{B,B})^{-1}K_{B,A}$, the Schur complement.  Using the equality $\det(K) = \det(K_{B,B})\cdot\det(K_{A,A} - K_{A,B}(K_{B,B})^{-1}K_{B,A})$, we obtain
\begin{align*}
\log \det(\Sigma_{A,A}) & = - \log \det (\Sigma_{A,A})^{-1} \\
& = - \log \det(K_{A,A} - K_{A,B}(K_{B,B})^{-1}K_{B,A}) \\
& = \log \det (K_{B,B}) - \log \det(K) .
\end{align*}
Combining this with Remark~\ref{rem:flip}, it follows that the polytopes given by $A \mapsto \log \det(\Sigma_{A,A})$ and by $A \mapsto \log \det(K_{A,A})$ are negatives of each other. 
\end{proof}

In other words, by using $K$ instead of $\Sigma$ we obtain the \emph{dual semigraphoid} defined in~\cite{Morton_et_al}.  In particular, if a semigraphoid has a faithful regular Gaussian distribution, then so does its dual.

\begin{example}[Multiinformation of the 4-node DAG in Example~\ref{ex:4nodes}]

We start by constructing $\Lambda$ from $\mathcal{G}$ using edge weights $1$ (i.e. $\ell_{ij}=1$ if $i\rightarrow j$ is an edge in $\mathcal{G}$). This choice of edge weights is sufficiently generic, since it results in a distribution that is faithful to $\mathcal{G}$. We then compute $\Lambda \Lambda^T$:

\begin{minipage}{.4\linewidth}
\[
\Lambda=\left(\begin{array}{cccc}
1 & 0 & -1 & 0\\
0 & 1 & -1 & 0\\
0 & 0 & 1 & -1\\
0 & 0 & 0 & 1\\
\end{array}\right)
\]
\end{minipage}
\begin{minipage}{.4\linewidth}
\[
K =  \Lambda \Lambda^T = 
\bgroup\begin{pmatrix}2&
     1&
     {-1}&
     0\\
     1&
     2&
     {-1}&
     0\\
     {-1}&
     {-1}&
     2&
     {-1}\\
     0&
     0&
     {-1}&
     1\\
     \end{pmatrix}\egroup
\]
\end{minipage}

\noindent Taking the log of the principal minors, we arrive at the system of inequalities:
%\begin{align*}
\begin{equation*}
\begin{split}
x_1&\le\log{2}\\
x_2&\le\log{2}\\
x_3&\le\log{2}\\
x_4&\le 0 \\
x_1+x_2&\le\log{3}\\
\end{split}
\quad\quad
\begin{split}
x_1+x_3&\le\log{3}\\
x_1+x_4&\le\log{2}\\
x_2+x_3&\le\log{3}\\
x_2+x_4&\le\log{2}\\
x_3+x_4&\le 0\\
\end{split}
\quad\quad
\begin{split}
x_1+x_2+x_3&\le\log{4}\\
x_1+x_2+x_4&\le\log{3}\\
x_1+x_3+x_4&\le 0 \\
x_2+x_3+x_4&\le 0 \\
x_1+x_2+x_3+x_4&= 0\\
\end{split}
\end{equation*}

%x_1,\,x_2,\,x_3,\,x_4,\, x_1+x_4,\, x_2+x_4 &\le \log 2 \\
%x_1+x_2,\, x_1+x_3,\, x_2+x_3,\, x_1+x_2+x_4 & \le \log 3 \\
%x_3+x_4,\, x_1+x_2+x_3 &\le \log\left(\frac{1}{2}\right) \\
%x_1+x_3+x_4,\, x_2+x_3+x_4 &\le \log\left(\frac{1}{4}\right)\\
%x_1+x_2+x_3+x_4 & =\log\left(\frac{1}{8}\right)
%\end{align*}
For instance, the submatrix $K_{\{1,3\},\{1,3\}}$ is $\begin{pmatrix} 2 & -1\\-1 &2 \end{pmatrix}$, whose determinant is $3$, giving the inequality $x_1 + x_3 \leq  \log{3} $. These inequalities give a realization of the DAG associahedron in Example~\ref{ex:4nodes}.  This realization is geometrically different from but has the same normal fan as the rational realization obtained by matroid polytopes that we will present in \S\ref{sec:MSMP}.
%We have verified using the heuristics below {\color{red} Is there a better way?} that these inequalities give a correct realization of the DAG associahedron in Example~\ref{ex:4nodes}, i.e.\ the choices of $\ell_{ij}$'s are generic.  This realization is geometrically different from but has the same normal fan as the rational realization obtained by matroid polytopes that we will see in \S\ref{sec:MSMP}.
\qed
\end{example}

As seen in the example above, there is a problem with this construction: 
%are two problems with this construction described above.  First, it is hard to prove that the choices of $\ell_{ij}$'s are generic enough. {\color{red} Caroline, is this true?}
%Second, 
The constant terms in the inequalities will almost never be all rational numbers, making it difficult to obtain exact combinatorial information such as the $f$-vector and the normal fan.  
We found that the following {\em heuristic} works well in practice to obtain exact combinatorial information from this polytope description:  First, round off the real numbers in the inequalities to nearby rational numbers (e.g.\ using 52 bit precision).  Then use exact arithmetic to compute the vertices of the polytope defined by these approximate inequalities.  This results in approximations of the true vertices.  Then form an approximate {\em slack matrix} by evaluating each approximate inequality at each approximate vertex and replace the entries in the slack matrix by $0$ or $1$ depending on whether the entry is approximately zero or not (e.g.\ by rounding off to 35 bit precision) to obtain an {\em incidence matrix} between the vertices and facets.  By eliminating duplicate rows and columns from this matrix we obtain the incidence matrix of a new polytope, from which its face lattice can be computed. In our simulations, the incidence matrix obtained this way gives the correct number of vertices and facets of the DAG associahedron, at least in small dimensions. However, this does {\em not} immediately lead to a rational realization of the polytope. Our code is available on Github at \url{https://github.com/foxflo/DAG-associahedra}.

It is possible that in some (or even all) instances we may be able to choose the parameters $\ell_{ij}$'s in such a way that the logarithms of the principal minors are all rational.  However, we do not know of any systematic way to do this, nor do we know of a systematic way to transform a nonrational realization into a rational realization.  We leave this as an open problem for future work. However, it is clear that if there is a nonrational realization, then there is also a rational realization, since a realization is a submodular function that satisfies some of the inequalities in \eqref{eqn:submodular} (those corresponding to the CI relations) at equalities and the rest as strict inequalities, and these linear constraints have rational coefficients.

To end this section, note that although for this paper it is sufficient to study the Gaussian setting, submodularity of the multiinformation holds for any probability distribution with finite multiinformation, which includes for example marginally continuous measures~\cite{Studeny}. Hence any set of CI relations that has a faithful realization by a distribution with finite multiinformation gives rise to a polytope similar to a DAG associahedron when contracting the edges correspond to CI relations in the permutohedron.

%% file: 5_MSMP.tex
\section{A construction of DAG associahedra as Minkowski sums of matroid polytopes}
\label{sec:MSMP}

In the following, we obtain a construction of DAG associahedra as Minkowski sums of matroid polytopes, resulting in a rational realization of these polytopes. Until now we viewed a semigraphoid as defined by CI relations. However, we can equivalently define a semigraphoid by its complementary conditional dependence relations.  Minkowski addition of generalized permutohedra translates to taking the union of the corresponding conditional dependence relations, since the union of normal cones to the edges of the Minkowski sum is the union of normal cones to the edges in the summand polytopes.\footnote{In other words, the tropical hypersurface of a Minkowski sum of polytopes is the union of the tropical hypersurfaces of individual polytopes.}  

For every dependence relation $i \notindependent j \mid K$ in the semigraphoid defined by a DAG $\mathcal{G}$, we wish to find a matroid whose semigraphoid, defined by its rank function as in (\ref{eqn:matroid}), contains the given relation and whose dependence relations are all valid for the semigraphoid of the DAG.  

We now describe how to construct these matroids. For any conditional dependence relation $i \notindependent j \mid K$ in the semigraphoid defined by a DAG~$\mathcal{G}$, there is a Bayes ball path from $i$ to $j$ given $K$.  
We first partition the Bayes ball path into {\em canyons} and  {\em treks} as follows.

\begin{defn} 
\label{def:trekcanyon}
A {\em trek} along a path is a consecutive subpath that does not contain any colliders.
A {\em canyon} along a path is a consecutive subpath that is palindromic with exactly one collider in the middle such that all edges are directed toward the collider.
A Bayes ball path is called {\em simple} if no node is repeated except in the same canyon and the maximal canyons are pairwise disjoint.
\end{defn}
If we think of the arrows as always pointing down, then a canyon is a path that first goes down and then backtracks up the same edges to the first node. See Figure~\ref{fig:8nodes} for an example. A single collider is a canyon by itself but not necessarily a maximal one.

The {\em active paths} in~\cite{Shachter}  can be obtained from simple Bayes ball paths by replacing each canyon with only the top of the canyon, e.g.\ for the Bayes ball path $1\,4\,8\,4\,3$ (given $\{8\}$) we get an active path $1\,4\,3$.  We prefer to keep the canyons in the path because we will need them for our matroid construction below.

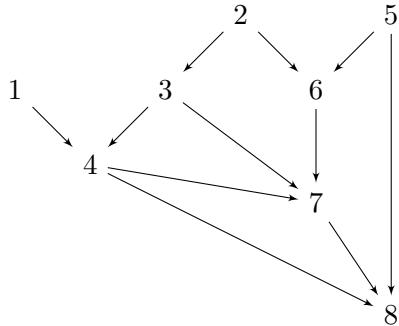
\begin{figure}[b!]
	\begin{tikzpicture}
	\tikzset{vertex/.style = {shape=circle,draw,minimum size=1.5em}}
	\tikzset{edge/.style = {->,> = latex'}}
	\node[draw=none,fill=none] (1) at (0,-1) {1};
	\node[draw=none,fill=none] (2) at (3,0) {2};
	\node[draw=none,fill=none] (3) at (2,-1) {3};
	\node[draw=none,fill=none] (4) at (1,-2) {4};
	\node[draw=none,fill=none] (5) at (5,0) {5};
	\node[draw=none,fill=none] (8) at (5,-4) {8};
	\node[draw=none,fill=none] (6) at (4,-1) {6};
	\node[draw=none,fill=none] (7) at (4,-2.5) {7};
	
	\draw[edge] (1) to (4);
	\draw[edge] (2) to (3);
	\draw[edge] (2) to (6);
	\draw[edge] (3) to (4);
	\draw[edge] (3) to (7);
	\draw[edge] (4) to (8);
	\draw[edge] (4) to (7);
	\draw[edge] (5) to (8);
	\draw[edge] (5) to (6);
	\draw[edge] (6) to (7);
	\draw[edge] (7) to (8);
	\end{tikzpicture}
\caption{In the DAG, $\overline{1} \rightarrow \underline{4} \leftarrow \overline{3 \leftarrow 2} \rightarrow \underline{6 \rightarrow 7 \leftarrow 6} \leftarrow \overline{5 \rightarrow 8}$ is a Bayes ball path from $1$ to $8$ given $\{4,7\}$. The treks and canyons along the path are overlined and underlined respectively.}
\label{fig:8nodes}
\end{figure}

\begin{lem}
\label{lem:simplePath}
If there is a Bayes ball path from $i$ to $j$ given $K$, then there is a simple one that is an alternating sequence of disjoint treks and canyons, starting and ending with treks.  
\end{lem}  

\begin{proof}Suppose there is a repeated node $a$.  Then we can take the first edge into $a$ and the last edge out of $a$.  This is allowed except when  $a$ would become a collider on the new path and $a$ is not in the conditioned set $K$. In this case there must be a descendant of $a$ that is a collider, hence in $K$, so we can make a canyon between $a$ and this collider.  The same argument shows that we can make the maximal canyons to be pairwise disjoint and that the end nodes $i$ and $j$ are not in any canyons.

On the simple path, each connected component of the maximal canyons and their adjacent edges is a trek by definition, since it does not contain any colliders.  Note that a single collider is considered a canyon.  There must be at least one collider, hence a canyon, between any two such treks.  For every canyon, we may assume that the node at the top must have two arrows pointing into it on the path; otherwise we can replace the canyon with just the top of the canyon to get another simple Bayes Ball path. If there are two consecutive canyons, then the edge between them cannot have an arrowhead at both canyons, so we can shortcut at least one of them.  Thus we may assume that canyons do not occur next to each other, i.e.\ any two canyons are separated by a trek.
\end{proof}

For example, in Figure~\ref{fig:8nodes} the Bayes ball path $1\,4\,8\,7\,4\,3$ from $1$ to $3$ given $\{8\}$ has a repeated node $4$, and simply removing the path between the two occurrences of $4$ would give $1\,4\,3$, which is {\em not} a Bayes ball path given $\{8\}$ since the collider $4$ is not in the conditioned set $\{8\}$. However, $4$ has a descendant, $8$, which is a collider in the original Bayes Ball path, so we can create a canyon $4 \rightarrow 8 \leftarrow 4$ and take the path $1\,4\,8\,4\,3$ instead.  

\noindent{\bf Construction of a matroid from a simple Bayes ball path.}

Let $\alpha$ be a simple Bayes ball path from $i$ to $j$ given $K$ which is an alternating sequence of treks and canyons as in Lemma~\ref{lem:simplePath}.   
Suppose we have $d+1$ treks $t_1,\dots,t_{d+1}$ and $d$ canyons $c_1,\dots,c_{d}$, in the order $t_1 c_1 t_2 c_2 \cdots c_{d} t_{d+1}$.  For $k=1,\dots,d$, let $M_i$ be the rank $2$ uniform matroid on $\{t_k, c_k, t_{k+1}\}$, i.e.\ a subset is independent if and only if it has size $\leq 2$.  It can be represented by affine independence among three distinct points on an affine line or as linear independence among three non-parallel vectors in a $2$-dimensional vector space or as edges of a triangle.  

Let $TC_\alpha$ be the matroid on the set of treks and canyons $\{t_1,c_1,\dots,t_d,c_d,t_{d+1}\}$, defined as the {\em parallel connection} or (free and proper) {\em amalgam} of these $k$ matroids along the treks~\cite[\S7, \S11]{Oxley}.  The parallel connection of two graphic matroids is obtained by gluing two graphs along an edge, which corresponds to a trek in our case.  The parallel connection of two affine independence matroids is obtained by placing the affine spaces in a common ambient affine space in such a way that they only intersect at one point, which corresponds to a trek in our case.  The matroid $TC_\alpha$ is constructed by repeating this operation, which is clearly associative.

Finally the matroid $M_\alpha$ on the node set $[n]$ of the DAG, is defined as follows.  
Let $TC'_\alpha$ be the matroid $TC_\alpha$ with an additional loop element $\ell$.  Let  $f : [n] \rightarrow TC'_\alpha$ be a function that sends each element on the path $\alpha$ to the trek or canyon containing it and all other elements to the loop $\ell$.  We say that a subset $S \subseteq [n]$ is independent in the matroid $M_\alpha$ if $\{f(a) : a \in S\}$ is independent in $TC'_\alpha$.  In particular, elements in the same trek or the same canyon become parallel elements (two-element circuits). Two examples of such matroids for different Bayes ball paths are shown in Figure~\ref{fig:matroids}.

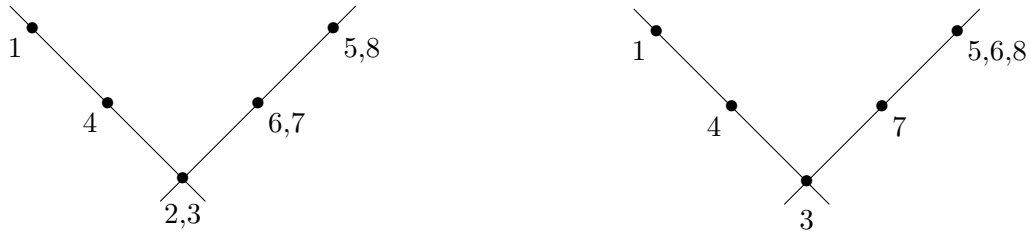
\begin{figure} [b]   
    \begin{minipage}{0.5\textwidth}
\begin{center}
	\begin{tikzpicture}
		\node[draw=none,fill=none] (1) at (0,0) {\textbullet};
		\node[draw=none,fill=none] (23) at (2,-2) {\textbullet};
		\node[draw=none,fill=none] (4) at (1,-1) {\textbullet};
		\node[draw=none,fill=none] (67) at (3,-1) {\textbullet};
		\node[draw=none,fill=none] (58) at (4,0) {\textbullet};
		
		\node[draw=none,fill=none] (1) at (0,0) [below left] {1};
		\node[draw=none,fill=none] (23) at (2,-2.5) {2,3};
		\node[draw=none,fill=none] (4) at (1,-1) [below left] {4};
		\node[draw=none,fill=none] (67) at (3,-1) [below right] {6,7};
		\node[draw=none,fill=none] (58) at (4,0) [below right] {5,8};
		
		\draw (-.3,.3) -- (2.3,-2.3);
		\draw (1.7,-2.3) -- (4.3,.3);
	\end{tikzpicture}
\end{center}
(a) The matroid corresponding to the Bayes ball path $\overline{1}\,\underline{4}\,\overline{3\,2}\,\underline{6\,7\,6}\,\overline{5\,8}$, which goes from $1$ to $8$ given $\{4,7\}$.\vspace{0.5cm}
        \end{minipage}
~\hspace*{0.9em}
        \begin{minipage}{0.45\textwidth}
\begin{center}
	\begin{tikzpicture}
		\node[draw=none,fill=none] (1) at (2,0) {\textbullet};
		\node[draw=none,fill=none] (3) at (4,-2) {\textbullet};
		\node[draw=none,fill=none] (4) at (3,-1) {\textbullet};
		\node[draw=none,fill=none] (7) at (5,-1) {\textbullet};
		\node[draw=none,fill=none] (568) at (6,0) {\textbullet};

		\node[draw=none,fill=none] (1) at (2,0) [below left] {1};
		\node[draw=none,fill=none] (3) at (4,-2.5) {3};
		\node[draw=none,fill=none] (4) at (3,-1) [below left] {4};
		\node[draw=none,fill=none] (7) at (5,-1) [below right] {7};
		\node[draw=none,fill=none] (568) at (6,0) [below right] {5,6,8};
		
		\draw (1.7,.3) -- (4.3,-2.3);
		\draw (3.7,-2.3) -- (6.3,.3);
	\end{tikzpicture}
\end{center}
(b) The matroid corresponding to the Bayes ball path $\overline{1}\,\underline{4}\,\overline{3}\,\underline{7}\,\overline{6\,5\,8}$, which goes from $1$ to $8$ given $\{4,7\}$.  The element $2$ is a loop in the matroid, i.e.\ $\{2\}$ is dependent.
        \end{minipage}
\caption{\label{fig:matroids} Two matroids that are compatible with the DAG in Figure~\ref{fig:8nodes}.} 
\end{figure}

A subset $S$ of a matroid is called a {\em flat} if $\rank(S \cup \{a\}) > \rank(S)$ for every $a \notin S$.  The intersection of two flats is a flat.  The {\em span} or the {\em closure} of a set is the smallest flat containing it.  More precisely
\[
\Span(S) = \{a : \rank(S \cup \{a\}) = \rank(S) \}.
\]  A subset $A \subseteq \{t_1,c_1,\dots,t_d,c_d,t_{d+1}\}$ is a flat in $TC_\alpha$ if and only if $A \cap \{t_k, c_k, t_{k+1}\}$ is a flat for each $k=1,\dots,d$~\cite[Proposition~11.4.13]{Oxley}.  This can also be checked directly from the realization of $TC_\alpha$ using affine/linear independence or graphs. Note that a subset of $\{t_k, c_k, t_{k+1}\}$ is a flat if and only if it has size $\neq 2$.  Flats of $M_\alpha$ are inverse images under $f$ of flats in $TC_\alpha'$.

It follows that a subset $S \subseteq M_\alpha$ is a flat if and only if it satisfies all of the following conditions:
\begin{enumerate}
\item[(F0)] $S$ contains all the loops (the nodes that are not on $\alpha$)
\item[(F1)] If an element of a trek or a canyon is in $S$, then all the other nodes in the same trek or canyon are also in $S$.
\item[(F2)] For each $k=1,\dots,d$, if $S$ intersects (thus contains) two out of three treks/canyons in $\{t_k,c_k,t_{k+1}\}$, then it also intersects (thus contains) the third.
\end{enumerate}
\qed

Recall from  \S\ref{sec:background} that the rank function of a matroid gives a collection of conditional dependence relations of the form $a \notindependent b \mid C$, where $a,b \in [n]$ and $C \subseteq [n]\setminus\{a,b\}$ satisfy the condition~\eqref{eqn:matroid}, namely
\begin{equation}
\rank(C) + 1 = \rank(Ca) = \rank(Cab) = \rank(Cb)
\label{eqn:matroid2}
\end{equation}

\begin{lem}
\label{lem:BBP}
Let $\mathcal{G}$ be a DAG and let $\alpha$ be a simple Bayes ball path from $i$ to $j$ given $K$ in $\mathcal{G}$.  Then the conditional dependence relations of the matroid $M_\alpha$ form a subset of the set of conditional dependence relations defined by the semigraphoid corresponding to $\mathcal{G}$.
\end{lem}

\begin{proof}
Suppose the relation $a \notindependent b \mid C$ comes from the matroid $M_\alpha$.  We wish to show that there is a Bayes ball path in $\mathcal{G}$ between $a$ and $b$ given $C$.
  The condition~(\ref{eqn:matroid2}) can be translated as
\begin{equation}
\label{eqn:matroid3}
\Span(C) \subsetneq \Span(Ca) = \Span(Cab) = \Span(Cb).
\end{equation}

%{\color{red} I can't figure out rev. 3's comment: What is missing in the construction of a simple Bayes ball path on page 13 is the explanation why a path is obtained in the end in which different cayons do not intersect. Indeed, can�t it be the case that the procedure of repetition-removal for one node a results in repetition of another node b, which was not repeated before? In other words, why the overall procedure does not cycle? My suggestion to solve this little technical obstacle is to consider the active path between the nodes which has the least number of involved edges and show that it has the desired property}

Let us first consider the case when $a$ and $b$ are in the same trek or in the same canyon.  Then $C$ cannot contain any element from the same trek/canyon; otherwise both $a$ and $b$ would be in $\Span(C)$, contradicting~\eqref{eqn:matroid3}.  Any two nodes in the same trek or the same canyon are connected by a Bayes ball path if no node is conditioned.  Thus there is a Bayes ball path between $a$ and $b$ given $C$, along $\alpha$. 

Now suppose that $a$ and $b$ are in different treks/canyons.  Then condition \eqref{eqn:matroid3} implies that $\Span(C)$, hence $C$, cannot contain any element in the treks/canyons containing $a$ or~$b$.   

We claim that $C$ does not intersect any trek that lies strictly between $a$ and $b$ along $\alpha$.  
Otherwise, if we compute $\Span(Ca)$ by adding to $\Span(C)$ nodes along the path starting at $a$, then the process would terminate (i.e.\ the conditions (F0),(F1),(F2) would be satisfied) at or before the trek that intersects $C$, before it reaches $b$. Thus $b \notin \Span(Ca)$, contradicting the condition $\Span(Ca)=\Span(Cb)$ in \eqref{eqn:matroid3}.  

Next we claim that $C$ intersects every canyon that lies strictly between $a$ and $b$ along $\alpha$.  Suppose $C$ does not intersect a canyon.  But we have already shown that $C$ does not intersect the next trek (which may contain $b$) after the canyon, on the path from $a$ to $b$ along $\alpha$.   Then as before, the computation of $\Span(Ca)$ stops before it reaches $b$; so again $b \notin \Span(Ca)$.

Putting everything together we conclude that $C$ intersects all the canyons and none of the treks that lie strictly between $a$ and $b$ along $\alpha$.  This means that the subpath between $a$ and $b$ along $\alpha$ forms a Bayes ball path given $C$ in $\mathcal{G}$, and the relation $a \notindependent b \mid C$ is valid for the semigraphoid of $\mathcal{G}$.
\end{proof}

A proof of our main result now follows.
\begin{proof}[Proof of Theorem~\ref{thm:main}]
For every conditional dependence relation $i \notindependent j \mid K$ in the semigraphoid corresponding to a DAG $\mathcal{G}$, we find a simple Bayes ball path $\alpha$ from $i$ to $j$ given $K$ and construct a matroid $M_\alpha$.  By Lemma~\ref{lem:BBP}, all the conditional dependence relations coming from $M_\alpha$ are among those in the semigraphoid corresponding to $\mathcal{G}$.  By taking all the matroids of the form $M_\alpha$ where $\alpha$ runs over all simple Bayes ball paths in $\mathcal{G}$, we obtain exactly all the conditional dependence relations that are valid in the semigraphoid of $\mathcal{G}$.  Taking the union of all these dependence relations translates into taking the Minkowski sum of all the corresponding matroid polytopes. Hence, DAG associahedra are obtained as Minkowski sums of matroid polytopes. 
\end{proof}

\begin{example}[Example~\ref{ex:4nodes} continued]
\label{ex:4nodesAgain}
Consider the DAG with 4 nodes from Example~\ref{ex:4nodes}.  The path $\overline{1}\,\underline{3\,4\,3}\,\overline{2}$ is a Bayes ball path from $1$ to $2$ given $\{4\}$.  The corresponding rank $2$ matroid is realized by affine dependence among $3$ distinct points in $\RR$.  For example, we can take the matroid on the points $v_1 = 1, v_2 = 2, v_3=v_4=3$ on the real line $\RR$ under affine independence.  Equivalently, the matroid is given by columns of the matrix $\begin{pmatrix} 1&1&1&1\\1&2&3&3 \end{pmatrix}$ under linear independence.  The bases are all pairs $\{v_i,v_j\}$ where $i\neq j$, except the pair $\{v_3,v_4\}$.  The matroid polytope is a square-based pyramid, with $e_1+e_2$ as the top of the pyramid.\qed
\end{example}
%\qed

%% file: 6_mixedGraphs.tex
\section{Generalizing to mixed graphs}
\label{sec:mixed}

The MSMP construction generalizes to a much more general setting of semigraphoids arising from loopless mixed graphs (LMG) introduced by Sadeghi and Lauritzen in~\cite{SadeghiLauritzen}.  We first recall the definitions.  A {\em mixed graph} is a graph with three possible types of edges: undirected ($i$\----$j$), directed ($i \longleftarrow j$ or $i \longrightarrow j$), or bidirected ($i \longleftrightarrow j$), which are also called {\em lines}, {\em arrows}, and  {\em arcs} respectively.  Multiple types of edges are allowed between any two nodes. A loopless mixed graph is a mixed graph without a loop, or an edge between a node and itself.

A node $j$ is called an {\em ancestor} of a node $i$, and $i$ is called a descendant of $j$, if there is a path $i=i_0,i_1,\dots,i_n=j$ from $i$ to $j$ in which the edges $(i_k, i_{k+1})$ are arrows (directed edges) pointing from $i_k$ to $i_{k+1}$ for all $k=0,\dots,n-1$.  Note that undirected and bidirected edges are not used in the definition of ancestors.  The set of ancestors of a node $i$ is denoted by $\an(i)$.  For any set of nodes $K$, let $\an(K) = \displaystyle\bigcup_{k \in K}\an(k)$.

For the path $i j k$ (where we may have $i=k$) the node $j$ is called a {\em collider} if the path is one of $i \longrightarrow j \longleftarrow k$, $i \longleftrightarrow j \longleftarrow k$, or $i \longrightarrow j \longleftrightarrow k$.  Otherwise $k$ is a noncollider.

Let $K$ be a subset of the node set of an LMG.  A path is called a {\em Bayes ball path given $K$} (called an {\em m-connecting path given $K$} in~\cite{SadeghiLauritzen}) if all its collider nodes are in $K \cup \an(K)$ and all its noncollider nodes are outside $K$.  We say that $i \not\independent j \mid K$ if there exists a Bayes ball path from $i$ to $j$ given $K$.  This collection of CI relations forms a graphoid~\cite{SadeghiLauritzen}.

We define {\em treks}, {\em canyons}, and {\em simple} Bayes ball paths in an LMG in exactly the same way as in Definition~\ref{def:trekcanyon}.  Only directed edges (arrows) can appear in a canyon, but all three types of edges are allowed on a trek.  A trek or a canyon may consist of only one node, but it may not be empty.

\begin{lem}
\label{lem:simpleLMG}
Let $i,j$ be nodes in an LMG and $K$ be a set of nodes such that $\{i,j\} \cap K = \varnothing$.  If there is a Bayes ball path from $i$ to $j$ given $K$, then there is a simple one that is a sequence of treks and canyons, starting and ending with treks, such that 
\begin{enumerate}
\item between any two treks there is at least one canyon
\item on the edge between a trek and a canyon, there must be an arrowhead at the canyon, and
\item two consecutive canyons can only be connected by a bidirected edge.
\end{enumerate}
\end{lem}

\begin{proof}

The existence of a simple path follows from the same argument as in the proof of Lemma~\ref{lem:simplePath}.   On the simple path, each connected component of the complement of maximal canyons (and adjacent edges) is a trek, since it does not contain any collider.  Note that a single collider is considered a canyon.  
The property (1) follows from the construction of treks as connected components.  Consider an edge connecting a trek and a canyon.  If there is no arrowhead at the canyon, then the top of the canyon does not have two arrowheads pointing into it.  We can then cut the canyon short, replacing the canyon with only the top of it, to get another simple Bayes ball path. Thus property (2) is satisfied.   An analogous argument shows that an edge between two canyons must have arrowheads at both canyons; otherwise we can cut the canyons short.  Thus property (3) is also satisfied.
\end{proof}

Now we can describe a generalization of the matroid construction from the previous section. 

\noindent{\bf Construction of a matroid from a Bayes ball path in an LMG.}
Let $\alpha$ be a simple Bayes ball path from $i$ to $j$ given $K$ satisfying the conditions from Lemma~\ref{lem:simpleLMG}.  
Suppose there are $d+1$ treks $t_1,\dots,t_{d+1}$ on $\alpha$, in this order.  For $k=1,\dots,d$, let $m_k$ denote the number of canyons between $t_k$ and $t_{k+1}$.
For each subpath $t_k c_{k,1} \cdots c_{k,m_k} t_{k+1}$ of two treks separated by canyons, consider the uniform matroid $U_{m_k+1,m_k+2}$ on $\{ t_k, c_{k,1}, \dots ,c_{k,m_k}, t_{k+1}\}$, which can be represented by affine independence among $m_k+2$ general points in $\RR^{m_k}$.  We then take the parallel connection of these uniform matroids along the treks.  In other words, we place the affine spaces, one for each pair of treks separated by a sequence of canyons, in a common ambient space so that any two consecutive ones only meet at one point and they affinely span maximum possible dimension. As before, a subset is a flat if and only if its intersection with each of the uniform matroids is also a flat.  

The matroid $M_\alpha$ on the node set $[n]$ of the LMG is defined as before by replacing each trek (resp. canyon) with parallel elements corresponding to the nodes in the trek (resp. canyon) and considering nodes not on $\alpha$ as loops.

A subset $S$ of the node set $[n]$ is a flat in $M_\alpha$ if and only if it satisfies (F0), (F1), and 
\begin{enumerate}
\item[(F$2'$)] For each $k = 1,\dots, d$, if S intersects (thus contains) $m_k+1$ out of $m_k+2$ treks/canyons in
$\{t_k, c_{k,1}, \dots, c_{k,m_k}, t_{k+1} \}$, then it also intersects (thus contains) the remaining one.
\end{enumerate}
\qed

For example, for the Bayes ball path $t_1 \longleftrightarrow c_{1,1} \longleftrightarrow c_{1,2} \longleftarrow t_2 \longleftrightarrow c_{2,1} \longleftrightarrow c_{2,2} \longleftarrow t_3$, we can take the following representation via affine independence in $\RR^4$:
\[
\begin{array}{c|ccccccc}
\text{trek or canyon} &t_1&c_{1,1}&c_{1,2}&t_2&c_{2,1}&c_{2,2}&t_3\\
\hline
&0&1&2&3&3&3&3\\
\text{representation}
&0&1&4&9&9&9&9\\
\text{as points in $\RR^4$}&0&0&0&0&1&2&3\\
&0&0&0&0&1&4&9\\
\end{array}
\]
We have a rank $3$ uniform matroid on the first four elements and another rank $3$ uniform matroid on the last four elements, meeting at a point $t_2$.

\begin{thm}[Generalization of Main Theorem]
\label{thm:mixedMSMP}
The semigraphoid of any loopless mixed graph is submodular.  The associated coarsening of the $S_n$-fan is the normal fan of the Minkowski sum of matroid polytopes corresponding to simple Bayes ball paths satisfying the conditions in Lemma~\ref{lem:simpleLMG}.
\end{thm}

\begin{proof}
The result and the proof of Lemma~\ref{lem:BBP} and the Proof of Theorem~\ref{thm:main} in \S\ref{sec:MSMP} can be repeated word for word, with the word DAG replaced by LMG and the condition (F2) replaced by (F$2'$).
\end{proof}

%% file: 7_relations.tex
\section{Relationship among families of semigraphoids}

Recall from Lemma~\ref{lem:submodular} that 
%a coarsening of the $S_n$ fan is polytopal if and only if the corresponding semigraphoid is submodular. 
%\cite[Theorem~6]{Morton_et_al} gives a bijection between semigraphoids and convex rank tests. In particular a submodular semigraphoid corresponds to a submodular rank test, which is furthermore equivalent to a submodular function on $\mathcal{P}([n])$, discussed in Lemma~\ref{lem:submodular} and \cite[Proposition~12 and Theorem~14]{Morton_et_al}. Thus, 
submodular functions on $2^{[n]}$ correspond to polytopal coarsenings of the $S_n$ fan, and these are exactly the normal fans of generalized permutohedra.

Graph associahedra and DAG associahedra are special classes of generalized permutohedra that are defined up to equivalence of normal fans. The former can be realized as Minkowski sums of standard simplices and the latter can be realized as Minkowski sums of matroid polytopes (MSMP). What additional classes of generalized permutohedra can be realized in this way?  
Since the standard simplices are matroid polytopes, MSS polytopes are also MSMP.

Unfortunately, this question seems difficult to answer in general. 
%By the discussion above, generalized permutohedra correspond to submodular functions. Thus, it is equivalent to consider the submodular cone of functions $f:\mathcal{P}([n])\rightarrow\mathbb{R}$. 
For $n=3,4,5$ respectively, the cone of submodular functions has 5, 37, and 117978 extreme rays of which only 5, 23, and 149, respectively correspond to (connected) matroid polytopes. It suffices to consider connected matroids because the direct sum of matroids corresponds to the Minkowski sum of the corresponding matroid polytopes. Thus the matroid polytope of a disconnected matroid, which is the direct sum of nontrivial matroids, is the Minkowski sum of the matroid polytopes of these direct summands. Although the structure of these extreme rays is unclear, it seems unlikely due to their sparsity that many submodular semigraphoids will arise in this way.

Another interesting class of semigraphoids are gaussoids \cite{Matus_Gaussoids}, an abstraction of regular Gaussian distributions in the language of CI relations; see \S\ref{sec:background}. Since we have seen that probabilistic graphical models can be faithfully realized by regular Gaussian distributions, another natural question is whether  all regular Gaussian models (also called representable gaussoids) or even all gaussoids are MSMP.  Our interest in gaussoids stems from Theorem~\ref{edge_thm}, where gaussoids are natural.  %{\color{red} any other reasons for emphasis on gaussoids?  Ref\#1 asked.}

Gaussoids appear to be incompatible with the MSMP construction. We have computationally verified that for $3 \leq n \leq 8$  no submodular semigraphoid corresponding to a connected matroid on $[n]$ is a gaussoid. Thus, none of the extreme matroidal rays of the submodular cone are gaussoids. 

Conversely, not all gaussoids, in fact not even all representable gaussoids, can be obtained via MSMP. For example, \cite[table~A.1]{Drton_Xiao} lists all Gaussian CI models on four variables (up to equivalence) and examples 19, 20, 34, 50, 51 are not MSMP. On the other hand, the CI relations corresponding to graphical models in this list all correspond to generalized permutohedra arising as MSMP. 

\begin{figure}[!t]
\centering
\includegraphics[scale=0.4]{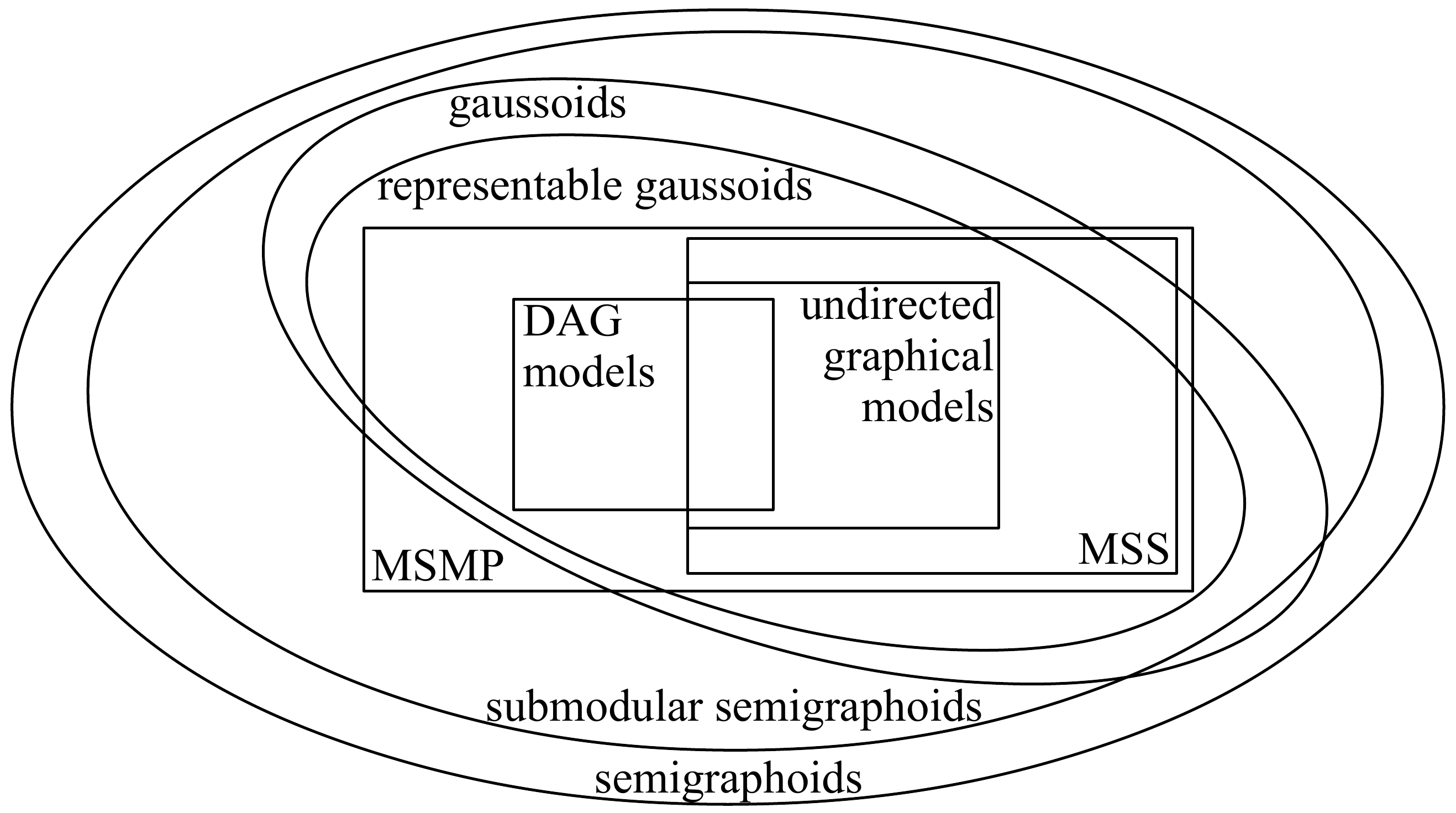}
\caption{\label{inclusions}  Venn diagram representing the relationship of all the different coarsenings of the $S_n$ fan discussed in this paper.}  
\end{figure}
%{\color{red} rev. 3 suggests mentioning that the intersection of DAG models and undirected graphical models gives the decomposable models.}

In Figure~\ref{inclusions} we illustrate the relationship of all the different coarsenings of the $S_n$ fan discussed in this paper by a Venn diagram. We have seen that undirected graphical models give rise to MSSs, while DAG models can be realized by MSMPs. In Proposition~\ref{prop:MSS-Vstructs} we showed that a DAG model is MSS if and only if it coincides with an undirected graphical model, i.e.~if and if it is a decomposable model. As we have discussed above, gaussoids are incompatible with the MSMP construction. In fact, gaussoids are also incompatible with the MSS construction. For example, it is easy to check that the standard simplex in Figure~\ref{fig:simplices}(a) is not a gaussoid. While every representable gaussoid is a submodular gaussoid as shown in Lemma~\ref{lem_Gaussian_multiinf}, this is not the case for gaussoids. The semigraphoid studied in~\cite[Section~3]{HMSSW} is a gaussoid that is not submodular.

%% file: 8_faces.tex
\section{Causal inference}
\label{sec:causalInference}

%
%Note also that since the vertices of graph associahedra correspond to tubings, or nested sets, the vertex posets are always rooted trees with the root being at the top. This is in general not the case for DAG associahedra; in fact all vertex posets of a DAG associahedron for $\mathcal{G}$ are rooted trees with the root at the top if and only if $\mathcal{G}$ is equivalent to an undirected graph.
%
%

In this section, we describe how DAG associahedra can be used to perform causal inference. The main problem in causal inference is the following:  We obtain data from an unobserved DAG $\mathcal{G}$. From this data we infer a set of CI relations $\mathcal{C}$. Under the faithfulness assumption, which we will assume throughout this section, $\mathcal{C}$ coincides with the gaussoid of $\mathcal{G}$.  The goal is to learn $\mathcal{G}$ from $\mathcal{C}$. This problem is ill defined since d-separation does not uniquely identify a DAG. So instead the problem is to learn $\mathcal{G}$ up to Markov equivalence, or in other words, to learn from $\mathcal{C}$ the {\em essential graph}, which is a partially directed graph with the same skeleton as $\mathcal{G}$ where an edge is directed if and only if it is directed the same way in every DAG in the Markov equivalence class.

A popular algorithm for learning the Markov equivalence class of a DAG is Greedy Equivalence Search (GES)~\cite{Meek_1997, Chickering_2002}, a greedy algorithm that searches through the space of DAGs by maximizing a scoring criterion such as the Bayesian Information Criterion (BIC). Under the faithfulness assumption GES is known to be consistent, i.e.~it learns the correct essential graph with probability approaching 1 as the sample size goes to infinity~\cite{Meek_1997, Chickering_2002}.  To reduce computation time, Teyssier and Koller~\cite{Koller_2005} suggested to replace the greedy search in DAG space by a greedy search in the space of all orderings; a scoring criterion such as BIC is optimized by performing a walk on the edges of the permutohedron. Although no consistency guarantees were given for this greedy algorithm, simulations suggest that the greedy ordering-based search has a similar performance and lower computational costs as compared to GES~\cite{Koller_2005}. In the following, we use our geometric insight on DAG associahedra to develop a new greedy ordering-based search with consistency guarantees.

Let $\cF$ be a coarsening of the $S_n$ fan. Each cone in $\cF$ is defined by inequalities of the form $x_i \leq x_j$ and can be labeled a poset on $[n]$.  
Then we get a map from permutations of $[n]$ to the set of partial orders on $[n]$, derived from the map sending a maximal $S_n$ cone to the maximal cone $\cF$ containing it.  The preimage permutation (total order) is a linear extension of its image partial order. Hence the maximal cones of the coarsened $S_n$ fan --- or the vertices of the generalized permutohedron if the fan is polytopal --- can be labeled by posets so that every permutation is a linear extension of exactly one of the posets.  If two permutations $\pi$ and $\tau$ are mapped to the same partial order, then we denote this by $\pi \sim \tau$.

A semigraphoid $\mathcal{C}$ on $[n]$ also gives a map from $S_n$ to the set of DAGs on nodes~$[n]$ as described in~\cite{SP_alg}:  To every permutation $\pi$ we associate a DAG $\mathcal{G}_\pi$ with 
\begin{equation}
\label{eqn:vertexDAG}
(\pi_i,\pi_j) \in \mathcal{G}_\pi \iff i < j \text{ and }  \pi_i \notindependent \pi_j \mid \{\pi_1, \dots , \pi_{\max(i,j)}\}\setminus \{\pi_i,\pi_{j}\}.
\end{equation}  
In other words, the edge directions in the graph must be compatible with the ordering $\pi = (\pi_1|\pi_2|\cdots|\pi_n)$, and the existence of an edge means that the two nodes are {\em not} independent given all the nodes that come before them in the ordering. $G_\pi$ is also known as  a minimal I-map or a directed independence graph.

We call $\pi$ a topological ordering of $\mathcal{G}$ if any edge $(i,j)$ in $\mathcal{G}$ implies that $i \succ j$ in $\pi$.
Note that if the semigraphoid comes from a DAG $\mathcal{G}$ and $\pi$ is a topological ordering of $\mathcal{G}$, then $\mathcal{G} = \mathcal{G}_\pi$.

\begin{algorithm}[!t]
\caption{Greedy SP algorithm on the permutohedron}
\label{alg:greedy_perm}
\begin{algorithmic}
\begin{STATE}

{\bf Input:} A set of CI relations $\mathcal{C}$ on $n$ random variables and a starting permutation $\pi\in S_n$
\vspace{0.2cm}

{\bf Output:} 
An essential graph $G$.

\begin{enumerate}

\item Set $t:=0$ and $\pi^{(0)} := \pi$. 

\item Set $t:=t+1$. Randomly select a permutation $\pi^{(t)}$ that differs from $\pi^{(t-1)}$ in a single adjacent transposition such that $\mathcal{G}_{\pi^{(t)}}$ is at least as sparse as $\mathcal{G}_{\pi^{(t-1)}}$.

\item Iterate (2) until convergence to the sparsest Markov equivalence class and output the corresponding essential graph.

\end{enumerate}
%\end{alg}

\end{STATE}
\end{algorithmic}

\end{algorithm}

In~\cite{SP_alg}, it was proposed to use the number of edges of $\mathcal{G}_\pi$ as a scoring criterion.  It was shown that an algorithm that outputs the Markov equivalence class of $\mathcal{G}_\pi$ with the fewest number of edges is consistent, i.e.~it outputs the correct Markov equivalence class, under strictly weaker conditions than faithfulness.  A permutation $\pi$ giving a sparsest DAG is called a {\em sparsest permutation}.  However the sparsest permutation (SP) algorithm is problematic from a computational point of view since it requires searching over all permutations. Instead, similarly as suggested in ~\cite{Koller_2005}, we can perform a greedy search by traversing the edges of the permutohedron, using the number of edges of $\mathcal{G}_\pi$ as a scoring function (see Algorithm~\ref{alg:greedy_perm}).

%\begin{thm}
%\label{thm:alg1}
%The Greedy SP algorithm is consistent under the faithfulness assumption.  
%\end{thm}
%We omit this proof since it is identical to the proof of Theorem~\ref{thm:alg2}, the consistency of the computationally more efficient Algorithm~\ref{alg:greedy_ass}.

Algorithm~\ref{alg:greedy_perm} requires searching through neighboring permutations even when they give rise to the same DAG. For example, the neighboring permutations $\pi=(1|2|3|4)$ and $\tau=(2|1|3|4)$ in Example~\ref{ex:4nodes} give rise to the same DAG $\mathcal{G}_\pi = \mathcal{G}_\tau=\mathcal{G}$ shown in Figure~\ref{graphs_ex_1} (left). We next discuss how to reduce the search space and hence computation time by performing the greedy search on the smaller DAG associahedron instead of the full permutohedron. %and keep the same consistency guarantees. 
The difficulty is that this needs to be done without having access to the DAG $\mathcal{G}$ on which the DAG associahedron is based. In order to do this, we give a description of the vertices and edges of a DAG associahedron in terms of the DAGs $\mathcal{G}_\pi$ that are associated to its vertices.

\begin{thm}
\label{vertex_thm}
For any fixed graphoid and two permutations $\pi$ and $\tau$, we have 
\[\pi \sim \tau \iff \mathcal{G}_\pi = \mathcal{G}_\tau.\]  Moreover, the equivalence class of $\pi$ consists of all topological orderings of $\mathcal{G}_\pi$.
\end{thm}

\begin{proof}
Suppose $\pi \sim \tau$. We may assume that 
\[
\pi = (a_1|\cdots|a_k|i|j|b_1|\dots|b_{n-k-2})\quad {\rm and}\quad  \tau= (a_1|\cdots|a_k|j|i|b_1|\dots|b_{n-k-2}),
\]
 where $i \independent j \mid \{a_1,\ldots,a_k\}$, since any pair of equivalent permutations is connected by a sequence of such pairs.  Now let us compare the edges in $\mathcal{G}_\pi$ and $\mathcal{G}_\tau$.  There is no edge between $i$ and $j$ in either DAG.  Between any two nodes in $[n]\backslash \{i,j\}$, it is clear that $\mathcal{G}_\pi$ and $\mathcal{G}_\tau$ coincide. 

Now suppose that $(a_\ell,j)$ is not an edge in $\mathcal{G}_\pi$ for some $\ell$.   Let $K = \{a_1,\dots,a_k\}\backslash\{a_\ell\}$. Then by applying the intersection property (INT) of graphoids from \S2 %\cite[Example~2.2 (2.5)]{Drton_Xiao}, 
we obtain
\begin{align*}
i \independent j \mid Ka_\ell \;\text{ and }\; j \independent a_\ell \mid Ki  
&  \stackrel{\text{(INT)}}{\implies}
%j \independent \{i,a_\ell\} \mid K
%\stackrel{\text{(SG2)}}{\implies}
j \independent a_\ell \mid K.  
\end{align*}
 Thus $(a_\ell,j)$ is not an edge in $\mathcal{G}_\tau$ either. Similarly if $(a_\ell,j)$ is not an edge in $\mathcal{G}_\tau$, then applying the semigraphoid property (SG2) we obtain%\cite[Example~2.2 (2.4)]{Drton_Xiao}
\begin{align*}
j \independent a_\ell \mid K \text{ and } i \independent j \mid K\cup\{a_\ell\}  & \stackrel{\text{(SG2)}}\implies
% j \independent \{i,a_\ell\} \mid K  
%\stackrel{\text{(SG3)}}{\implies}
 j \independent a_\ell \mid Ki.
\end{align*}
  Thus $(a_\ell,j)$ is not an edge in $\mathcal{G}_\pi$ either.  

We can check in a similar fashion by setting $K = \{a_1,\dots,a_k\}$ that for any $b_\ell \in [n] \backslash(\{a_1,\dots,a_k\}\cup\{i,j\})$ the edge $(j,b_\ell)$ is in $\mathcal{G}_\pi$ if and only if it is in $\mathcal{G}_\tau$. The same claims also hold for $i$ by switching $\pi$ and $\tau$.%Suppose $(j,b_\ell)$ is not an edge in $\mathcal{G}_\pi$ for some $b_\ell\in [n] \backslash(K\cup\{i,j\})$. Let $B=\{b_1,\ldots,b_{\ell-1}\}$. Then using the graphoid axioms from Section~2, we get
%\[
%i \independent j \mid K \text{ and } j \independent b_\ell \mid Ki\cup B
%\implies j \independent B\cup\{i,b_\ell\} \mid K
%\implies j \independent b_\ell \mid Ki\cup B
%\]
%where the two implications follow from (SG4) and (SG3) respectively. Thus $(j,b_\ell)$ is not an edge in $\mathcal{G}_\tau$ either, as desired.

For the converse, suppose $\tau$ is a topological ordering of $\mathcal{G}_\pi$.  In particular, this holds when $\mathcal{G}_\pi = \mathcal{G}_\tau$.    We wish to prove that $\pi \sim \tau$.  Without loss of generality we may assume that $\tau = (1|2|\cdots|n)$.  Let $\pi = (\pi_1|\pi_2|\cdots|\pi_n)$.  If $\pi \neq \tau$, then there is an $i \in [n-1]$ such that $\pi_i > \pi_{i+1}$.  Since $\pi_i$ and $\pi_{i+1}$ appear with opposite orders in $\pi$ and $\tau$ and $\tau$ is a topological ordering of $\mathcal{G}_\pi$, there is no edge between $\pi_i$ and $\pi_{i+1}$ in $\mathcal{G}_\pi$.  By construction of $\mathcal{G}_\pi$, we must have $\pi_i \independent \pi_{i+1} \mid \{\pi_1,\dots,\pi_{i-1}\}$ in the graphoid.  Let $\pi' = (\pi_1|\cdots|\pi_{i-1}|\pi_{i+1}|\pi_{i}|\pi_{i+2}|\cdots|\pi_n)$.  Then $\pi' \sim \pi$ by definition, so $\mathcal{G}_{\pi'} = \mathcal{G}_{\pi}$ as shown above.  Since $\tau$ is also a topological ordering of $\mathcal{G}_{\pi'}$, the statement $\tau \sim \pi$ follows by induction on the number of inversions in $\pi$.
\end{proof}

%Each vertex of the permutohedron is labeled by a permutation, which in conjunction with the set of CI relations $\mathcal{C}$ corresponds to a DAG. We first show that a greedy algorithm where we start in any vertex of the permutohedron and randomly pick a neighboring vertex that does not increase the number of edges in the corresponding DAG will converge to the true Markov equivalence class under the faithfulness assumption. Finally, we will show that we can reduce the search space by performing a greedy algorithm on the smaller DAG associahedron instead of the full permutohedron without losing the consistency guarantees.

In the following example we illustrate Theorem~\ref{vertex_thm} and show how the vertices of a DAG associahedron can be labeled by posets or by DAGs.

\begin{example}
\label{ex:4nodesB}
We return to Example~\ref{ex:4nodes}. Compared to the permutohedron, the DAG associahedron corresponding to $\mathcal{G}$ has six new vertices, namely:
\begin{enumerate}
\item[(a)] $(1|2|3|4), (2|1|3|4)$,
\item[(b)] $(1|2|4|3), (2|1|4|3)$,
\item[(c)] $(1|3|2|4), (1|3|4|2)$,
\item[(d)] $(2|3|1|4), (2|3|4|1)$,
\item[(e)] $(3|4|1|2), (3|1|4|2), (3|1|2|4)$,
\item[(f)] $(3|4|2|1), (3|2|4|1), (3|2|1|4)$.
\end{enumerate}
The posets representing these vertices and the corresponding DAGs are shown in Figure~\ref{posets_1}. Each of the other vertices of the DAG associahedron corresponds to a single permutation and the corresponding DAG has no missing edges. \qed

\begin{figure}[]
\centering
\begin{minipage}{0.15\textwidth}{\includegraphics[scale=0.42]{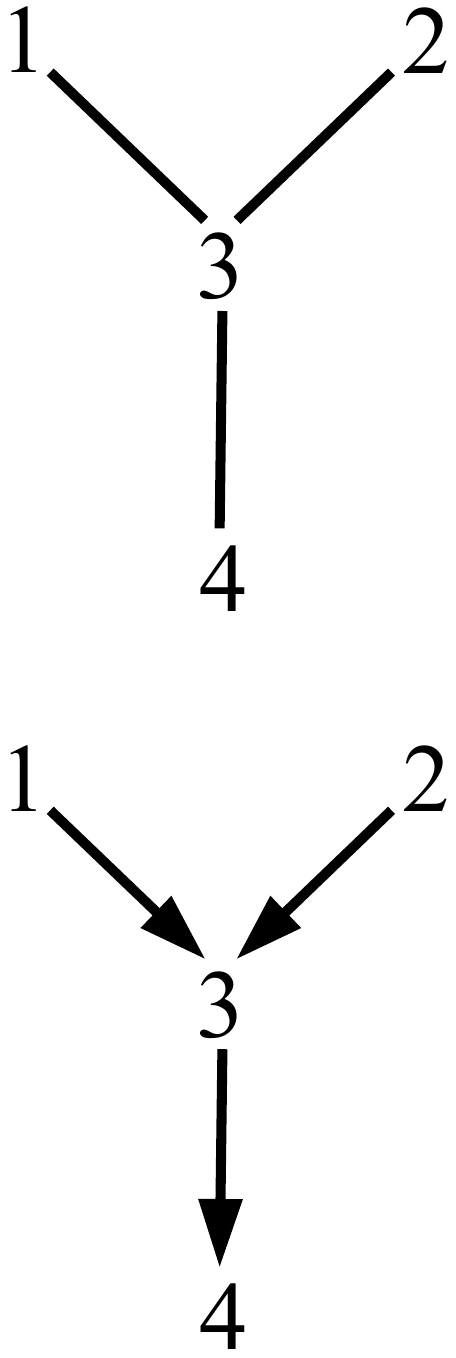}}\end{minipage}
\begin{minipage}{0.15\textwidth}{\includegraphics[scale=0.42]{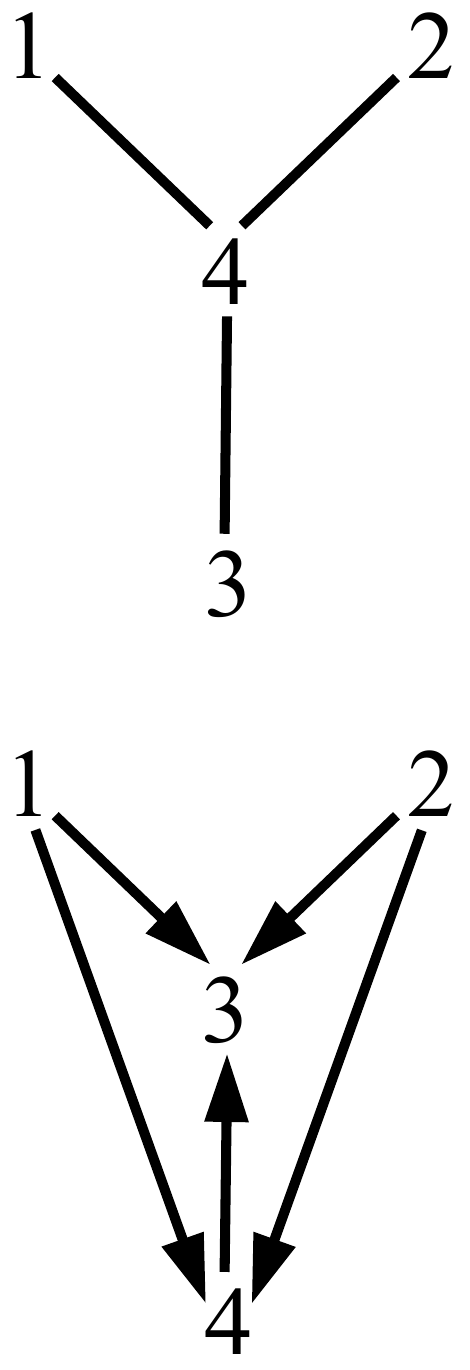}}\end{minipage}
\begin{minipage}{0.15\textwidth}{\includegraphics[scale=0.42]{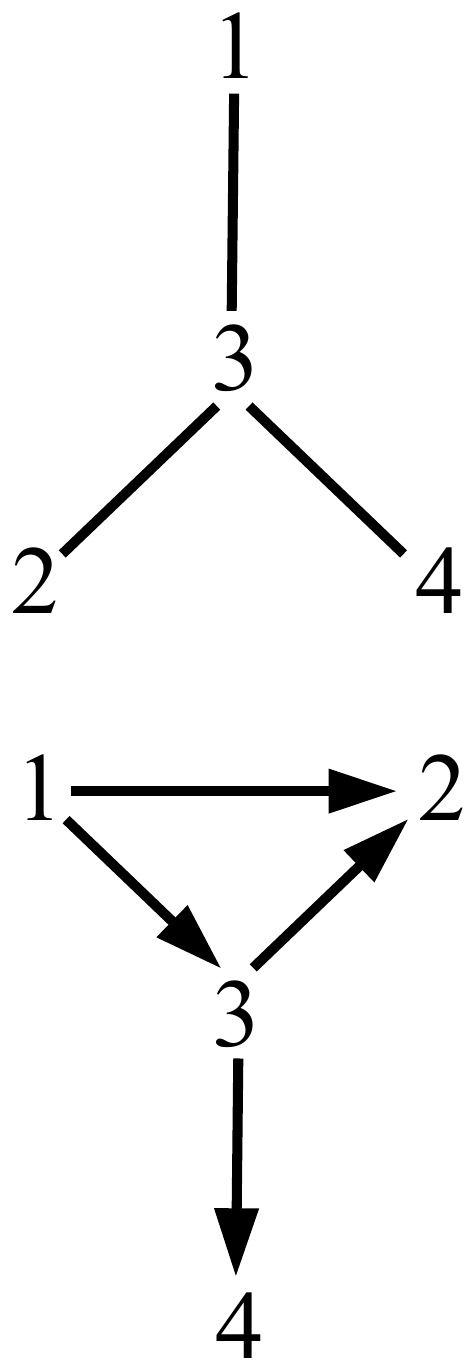}}\end{minipage}
\begin{minipage}{0.15\textwidth}{\includegraphics[scale=0.42]{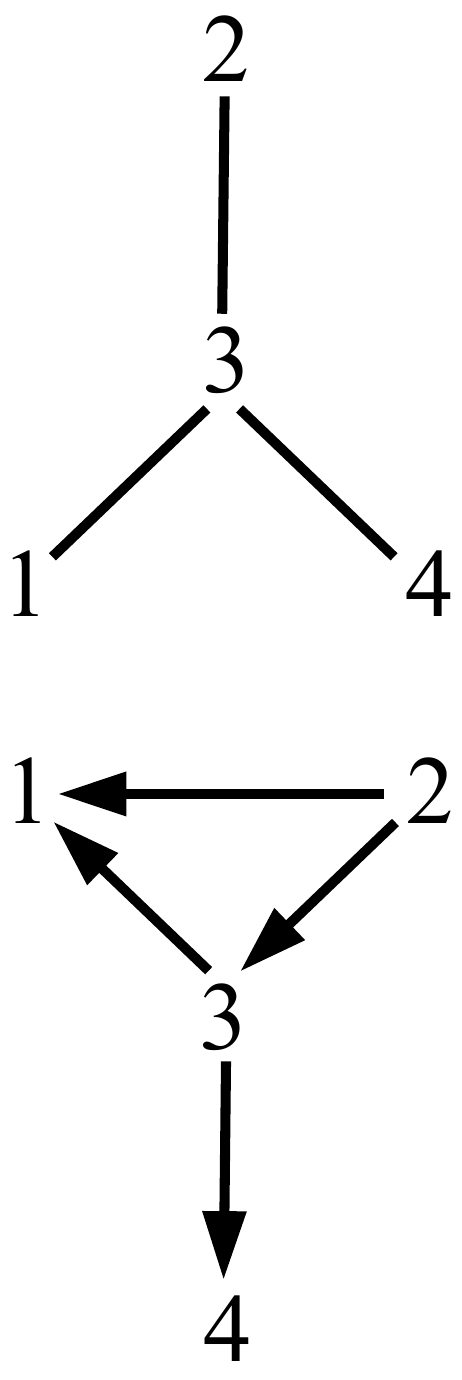}}\end{minipage}
\begin{minipage}{0.15\textwidth}{\includegraphics[scale=0.42]{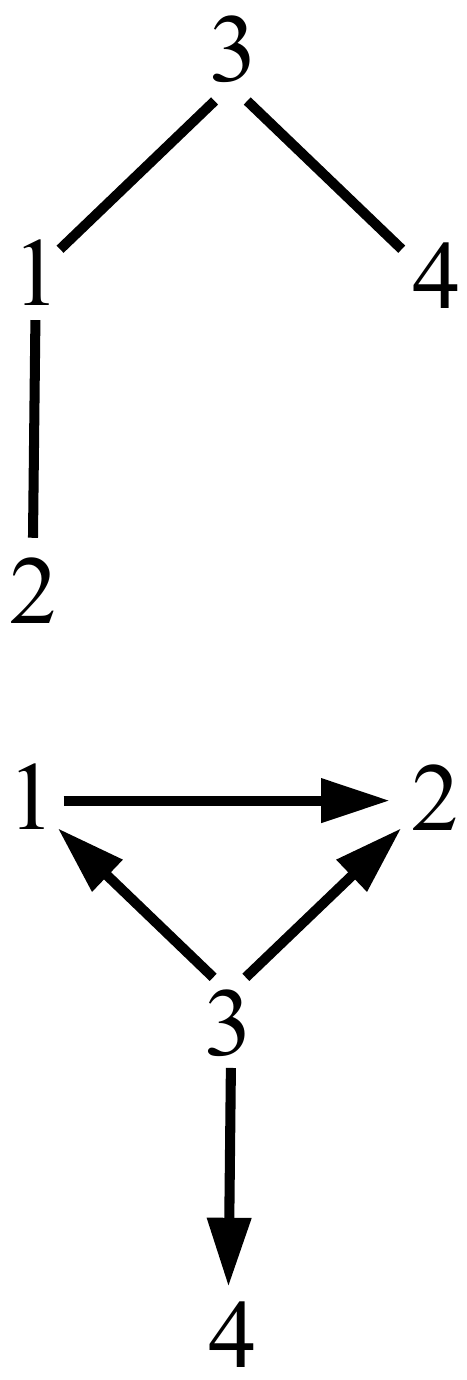}}\end{minipage}
\begin{minipage}{0.15\textwidth}{\includegraphics[scale=0.42]{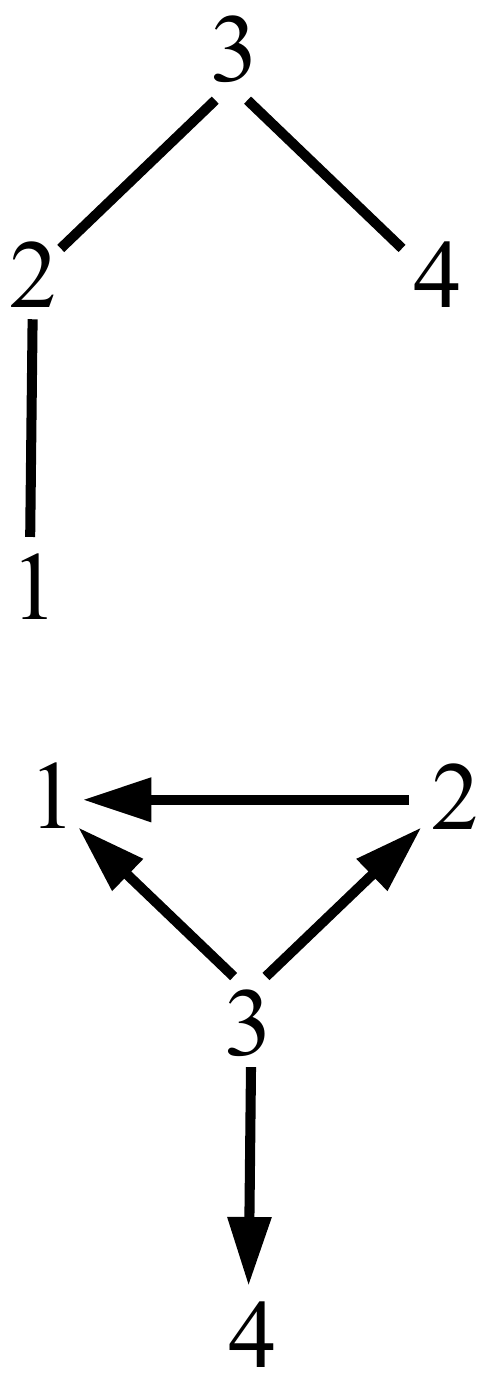}}\end{minipage}
\caption{\label{posets_1}Posets and their corresponding DAGs representing the new (compared to the permutohedron) vertices of the DAG associahedron discussed in Examples~\ref{ex:4nodes} and~\ref{ex:4nodesB}.}
\end{figure}

%Note that the faces of the DAG associahedron can be represented by a poset that is attained by intersecting the posets corresponding to all vertices that are adjacent to the face. 

\end{example}

If we have a description of the vertices of a DAG associahedron in terms of posets, then we know the maximal cones in the normal fan, so we can directly obtain all other normal cones by intersecting the maximal cones. In the following, we give an alternative description of the edges of the DAG associahedron in terms of the DAGs $\mathcal{G}_\pi, \mathcal{G}_\tau$ corresponding to the vertices adjacent to an edge $(\pi, \tau)$.

Chickering~\cite{Chickering_95} introduced the notion of a covered edge: a directed edge $(i,j)$ in $\mathcal{G}$ is \emph{covered} if $$\textrm{pa}(i) = \textrm{pa}(j)\setminus\{i\}.$$
We denote by $\overline{\mathcal{G}}$ the skeleton of a DAG $\mathcal{G}$. In addition, for two undirected graphs $G$ and $G'$ we say that $G$ is a subset of $G'$, i.e., $G\subseteq G'$, if $G$ and $G'$ have the same node set and every edge in $G$ is also an edge in $G'$.  

The following result shows that given a DAG label of a vertex of a DAG associahedron, we can find neighboring vertices whose underlying graph is not bigger by flipping the direction of a covered edge.  We will prove this result more generally for gaussoids. % defined in \S\ref{sec:background}.

\begin{thm}
\label{edge_thm}
Let $\cF$ be a coarsened $S_n$ fan corresponding to a gaussoid.  Suppose the equivalence classes of $\pi = (\pi_1|\pi_2|\cdots|\pi_n)$ and  $\tau = (\pi_1|\pi_2|\cdots|\pi_{i+1}|\pi_i |\cdots|\pi_n)$ are adjacent maximal cones in $\cF$.  Then  $\overline{\mathcal{G}}_{\tau}\subseteq \overline{\mathcal{G}}_{\pi}$ if and only if $(\pi_i,\pi_{i+1})$ is a covered edge in $\mathcal{G}_{\pi}$.
\end{thm}

\begin{proof}
First, note that $(\pi_i, \pi_{i+1})$ is an edge in $\mathcal{G}_\pi$, since otherwise $\mathcal{G}_\pi = \mathcal{G}_{\tau}$ by Theorem~\ref{vertex_thm}. We now prove the ``if'' direction. Without loss of generality we assume that $\pi=(1|2|\cdots|n)$, $\tau=(1|2|\cdots|i-1|i+1|i|i+2|\cdots|n)$ and $(i,i+1)$ is a covered edge in $\mathcal{G}_\pi$. Note that from the definition of $\mathcal{G}_\pi$ and $\mathcal{G}_\tau$ the only difference between these two DAGs can be in the presence or absence of edge $(\ell,i)$ or $(\ell,i+1)$ with $\ell<i$. In order to prove that $\overline{\mathcal{G}_{\tau}}\subseteq\overline{\mathcal{G_{\pi}}}$, we need to show that any missing edge $(\ell,i)$ or $(\ell,i+1)$ in $\mathcal{G}_{\pi}$ is also not present in $\mathcal{G}_\tau$. Now suppose that $(\ell,i)$ is a missing edge in $\mathcal{G}_{\pi}$ for some  $\ell<i$. Since the edge $(i,i+1)$ is covered in $\mathcal{G}_\pi$, then $(\ell,i+1)$ is also a missing edge in $\mathcal{G}_{\pi}$. Let $K=\{1,\ldots,i-1\}\backslash\{\ell\}$. By the definition of $\mathcal{G}_\pi$ and $\mathcal{G}_\tau$ we get that%Markov property implies that 
\[
\ell\independent i\mid K\quad\text{and}\quad \ell\independent i+1\mid Ki,
\]
and hence by the semigraphoid property (SG2) we obtain that $\ell\independent i+1 \mid K$ and $\ell\independent i \mid K\cup\{i+1\}$. Therefore, $(\ell,i)$ and $(\ell,i+1)$ are also missing edges in $\mathcal{G}_{\tau}$, and we conclude that $\overline{\mathcal{G}_{\tau}}\subseteq\overline{\mathcal{G_{\pi}}}$.

%Then we get $\ell\independent \{i,i+1\}\mid K$ from (SG4).  Applying the semigraphoid properties (SG3) and (SG2) we obtain
%\[
%\ell\independent \{i,i+1\}\mid K   \stackrel{\text{(SG3)}}{\implies}    \ell\independent i\mid K\cup\{i+1\}
%\quad\text{and}\quad\ell\independent \{i,i+1\}\mid K   \stackrel{\text{(SG2)}}{\implies}\ell\independent i+1 \mid K.
%\]  

For the ``only if'' direction suppose that $\overline{\mathcal{G}_\tau} \subseteq \overline{\mathcal{G}_\pi}$.  We want to show that the edge $(\pi_{i},\pi_{i+1})$ is a covered edge in $\mathcal{G}_{\pi}$.  Assume on the contrary that it is not. 

We first consider the case when there is an $a < i$ with $(a,i+1) \in \mathcal{G}_\pi$ but $(a,i) \notin \mathcal{G}_\pi$.  Then $(a,i) \notin \mathcal{G}_\tau$, and hence
\begin{equation}
\label{eqn:A}
a \ind i \mid K \cup \{i+1\},
\end{equation}
where $K=\{1,\dots,i-1\}\setminus\{a\}$.  We claim that $(a,i+1) \in \mathcal{G}_\tau$.  Otherwise we would have $a \independent i+1 \mid K$, which together with (\ref{eqn:A}) implies %$a \ind \{i,i+1\} \mid K$ by (SG4), which in turns implies 
$a \ind i+1 \mid Ki$ by (SG2), contradicting $(a,i+1) \in \mathcal{G}_\pi$.
From $(a,i+1) \in \mathcal{G}_\tau$, we have 
\begin{equation}
\label{eqn:B}
a \notindependent i+1 \mid K.
\end{equation}
Next we claim that 
\begin{equation}
\label{eqn:C}
%i \notindependent \{i+1, a\} \mid K.
i \notindependent i+1 \mid K.
\end{equation}
Otherwise, together with (\ref{eqn:A}) we would have $i \ind i+1 \mid Ka$ by (SG2), contradicting the assumption that $\pi$ and $\tau$ lie in adjacent cones of the fan $\cF$. % From (SG4), (\ref{eqn:A}), and (\ref{eqn:C}), we get
%\begin{equation}
%\label{eqn:D}
%i \notindependent i+1 \mid K.
%\end{equation}
Finally, from the weak-transitivity axiom (G2) for gaussoids we obtain
\begin{equation}
\label{eqn:E}
i \notindependent i+1 \mid K\quad\text{and}\quad a \notindependent i+1 \mid K \stackrel{\text{(G2)}}{\implies} a \notindependent i \mid K \quad\text{or}\quad a \notindependent i \mid K \cup \{i+1\}
\end{equation}
Combining~(\ref{eqn:A}) %,(\ref{eqn:B}),(\ref{eqn:D}), 
and (\ref{eqn:E}), we obtain 
$
a \notindependent i \mid K,
$
that is, $(a,i) \in \mathcal{G}_\pi$, contradicting the assumption that $(a,i) \notin \mathcal{G}_\pi$.

Now we consider the case where $(a,i) \in \mathcal{G_\pi}$, but $(a,i+1) \notin \mathcal{G_\pi}$, so $(a, i+1) \notin \mathcal{G_\tau}$.  Then
\begin{equation}
\label{eqn:F}
a \ind i+1 \mid Ki\quad \text{and}\quad a \ind i+1 \mid K.
\end{equation}
By the gaussoid axiom (G1), we have
\begin{equation}
\label{eqn:G}
i+1 \notindependent a \mid Ki\quad\text{or}\quad i+1 \notindependent i \mid Ka \;\Rightarrow\; a \notindependent i+1 \mid K\quad \text{or}\quad i \notindependent i+1 \mid K.
\end{equation}
Since $\pi$ and $\tau$ are in adjacent cones of the fan $\cF$, we have $i+1 \notindependent i \mid Ka$, so by (\ref{eqn:F}) and (\ref{eqn:G}),
\begin{equation}
\label{eqn:H}
i \notindependent i+1 \mid K.
\end{equation}
%We first claim that $(a,i) \in \mathcal{G}_\tau$; otherwise we would have $a \ind i \mid K \cup \{i+1\}$, and applying the intersection property (INT) to this and $a \ind i+1 \mid Ki$ in (\ref{eqn:F}) gives us $a \independent i \mid K$, contradicting $(a,i)\in \mathcal{G}_{\pi}$.  Thus $\mathcal{G}_\tau$ contains the edge $(a,i)$, which implies $a \notindependent i \mid K$.  This together with~(\ref{eqn:H}) and weak transitivity (G2) gives us 
Since by assumption $(a,i) \in \mathcal{G_\pi}$, then $a \notindependent i \mid K$.  This together with~(\ref{eqn:H}) and weak transitivity (G2) gives us 
%\[
$a \notindependent i+1 \mid Ki$ or $a \notindependent i+1 \mid K$,
%\]
which contradicts (\ref{eqn:F}).
\end{proof}

This result directly gives rise to an improved version of Algorithm~\ref{alg:greedy_perm}, which corresponds to performing a greedy search on the DAG associahedron instead of the permutohedron and does not require knowing the underlying true DAG (see Algorithm~\ref{alg:greedy_ass}). In this algorithm, we are given a set of CI relations $\mathcal{C}$ that are induced from a fixed but unknown DAG $\mathcal{G}$. In each iteration the algorithm outputs an auxiliary DAG, whose skeleton contains the skeleton of $\mathcal{G}$. 
%there is one fixed ``unknown'' DAG $\mathcal{G}$ which induces CI relation $\mathcal{C}$ while iterations of the algorithms are (different) auxiliary DAGs, whose skeleton contains the skeleton of $\mathcal{G}$
%We end by providing a sketch of the proof that this algorithm converges to $\mathcal{G}$ under the faithfulness assumption. The complete proof can be found in~\cite{greedy_SP}, a statistical follow-up work, where we show the importance of the geometric results obtained in this paper for applications to causal inference.
In a statistical follow-up work~\cite{greedy_SP} we show the importance of the geometric results obtained in this paper for applications to causal inference. In particular, we prove that Algorithm~\ref{alg:greedy_ass} is consistent under the faithfulness condition, i.e., that it converges to $\mathcal{G}$ under the faithfulness assumption. We end by providing a sketch of the proof. Let $\mathcal{G}$ denote the true DAG. Then $\mathcal{G} = \mathcal{G}_\pi$ for some $\pi$ (any topological ordering of $\mathcal{G}$).  Let $\tau \in S_n$. Then every independence relation that holds for $\mathcal{G}_\tau$ also holds for $\mathcal{G}$~\cite[Lemma~2.1]{SP_alg}. This implies $\overline{\mathcal{G}}\subseteq\overline{\mathcal{G}_\tau}$. If a permutation $\pi$ differs from $\tau$ only in the reversal of a covered edge in $\mathcal{G}_\tau$, then by Theorem~\ref{edge_thm} we have $\overline{\mathcal{G}_\pi} \subseteq \overline{\mathcal{G}_\tau}$.  At a high level, the proof follows from a result by Chickering~\cite[Theorem 4]{Chickering_2002} which says that using such edge reversals one can go from any DAG $\mathcal{G_\tau}$ to any DAG $\mathcal{G}_\pi$ with $\overline{\mathcal{G}_\pi} \subseteq \overline{\mathcal{G}_\tau}$. The difficulty lies in showing that there exists such a Chickering sequence which corresponds to a walk on the DAG associahedron, which is proven in~\cite{greedy_SP}.

\begin{algorithm}[!t]
\caption{Greedy SP algorithm on the DAG associahedron}
\label{alg:greedy_ass}
\begin{algorithmic}
\begin{STATE}

{\bf Input:} A set of CI relations $\mathcal{C}$ on $n$ random variables and a starting permutation $\pi\in S_n$
\vspace{0.2cm}

{\bf Output:} 
An essential graph $G$.

\begin{enumerate}

\item Set $t=0$ and $\pi^{(0)} = \pi$. 

\item Set $t:=t+1$. Randomly select a covered edge $(\pi^{(t-1)}_i, \pi^{(t-1)}_j)$ in $\mathcal{G}_{\pi^{(t-1)}}$ and reverse its direction. Let $\pi^{(t)}$ denote the resulting permutation and $\mathcal{G}_{\pi^{(t)}}$ the corresponding DAG.

%=(\pi^{(t-1)}(1), \dots, \pi^{(t-1)}(i-1), \pi^{(t-1)}(j), \pi^{(t-1)}(i+1), \dots ,  \pi^{(t-1)}(j-1),  \pi^{(t-1)}(i),  \pi^{(t-1)}(j+1),  \pi^{(t-1)}(n)).$$

\item Iterate (2) until convergence to the sparsest Markov equivalence class and output the corresponding essential graph.

\end{enumerate}
%\end{alg}

\end{STATE}
\end{algorithmic}

\end{algorithm}

%\begin{thm}
%\label{thm:alg2}
%Algorithm~\ref{alg:greedy_ass} is consistent under the faithfulness condition.
%\end{thm}
%
%\begin{proof}
%Let $\mathcal{G}$ denote the true DAG. Then $\mathcal{G} = \mathcal{G}_\pi$ for some $\pi$ (any topological ordering of $\mathcal{G}$).  Let $\tau \in S_n$. Then every independence relation that holds for $\mathcal{G}_\tau$ also holds for $\mathcal{G}$~\cite[Lemma~2.1]{SP_alg}. This implies $\overline{\mathcal{G}}\subseteq\overline{\mathcal{G}_\tau}$. If a permutation $\pi$ differs from $\tau$ only in the reversal of a covered edge in $\mathcal{G}_\tau$, then by Theorem~\ref{edge_thm} we have $\overline{\mathcal{G}_\pi} \subseteq \overline{\mathcal{G}_\tau}$.  At a high level, the proof follows from a result by Chickering~\cite[Theorem 4]{Chickering_2002} which says that using such edge reversals one can go from any DAG $\mathcal{G_\tau}$ to any DAG $\mathcal{G}_\pi$ with $\overline{\mathcal{G}_\pi} \subseteq \overline{\mathcal{G}_\tau}$. The difficulty lies in showing that there exists such a Chickering sequence which corresponds to a walk on the DAG associahedron. This is proven in~\cite[Theorem 17]{greedy_SP}.
%\end{proof}

%\subsection{Greedy SP algorithm on the DAG associahedron}

%% file: definitions.tex
\section{Polytopes and fans}
\label{sec:def}

Most of the following definitions can be found in \cite{Ziegler}. 
A \emph{polyhedron} is a subset of a real vector space $\R^n$ defined by finitely many linear inequalities. A \emph{polytope} is a bounded polyhedron.  Equivalently, a polytope is the convex hull of a finite set of points in $\mathbb{R}^n$. The \emph{Minkowski sum} of two polyhedra $P$ and $Q$ is defined as $P+Q=\{x+y\mid x\in P \text{ and } y\in Q\}$.  A (polyhedral) \emph{cone} is a polyhedron that is closed under addition and scaling by a nonnegative real number.  A \emph{face} of a polyhedron $P$ is a subset of $P$ that maximizes some linear functional. A face $F$ of a nonempty polyhedron $P$ is \emph{proper} if $F\neq P$. A \emph{facet} is an inclusion maximal proper face of $P$.

A \emph{fan} is a family $\mathcal{F}$ of nonempty polyhedral cones such that
\begin{enumerate}
\item every face of a cone in $\mathcal{F}$ is also a cone in $\mathcal{F}$;
\item the intersection of any two cones in $\mathcal{F}$ is a face of both.
\end{enumerate}
A fan in $\RR^n$ is {\em complete} if the union of its cones is equal to $\RR^n$.  A {\em wall} in a complete fan in $\RR^n$ is an $(n-1)$-dimensional cone in the fan.

For each face $F$ of $P$, the \emph{outer normal cone} $N_F$ is the set of all linear functionals that are maximized on $F$, i.e.
\[N_F = \{c\in(\mathbb{R}^n)^*\mid  F\subseteq \{x\in P\mid c \cdot x = \max_{y\in P} (c \cdot y)\}\}.\]
 The outer \emph{normal fan} of a polytope $P$ is the collection $\{N_F : F \text{ is a face of } P\}$, which is a complete fan in $(\RR^n)^*$.  We identify $(\R^n)^*$ and $\R^n$ using the usual dot product.   The \emph{inner normal} cones and fans are defined analogously by replacing ``max'' with ``min''. For two faces $F$ and $F'$ of $P$, we have $F\subseteq F'$ if and only if $N_F\supseteq N_{F'}$. In particular, the smallest cone in the normal fan of $P$ is a linear space, namely the orthogonal complement of $P$. The normal cones of facets are inclusion-minimal cones that strictly contain the smallest cone.  The maximal full-dimensional cones in the normal fan are normal cones of the vertices of $P$.

%In particular, the rays of the normal fan are the facet normal vectors of $P$, and the full dimensional cones in the normal fans are normal cones of the vertices of $P$.

%% file: submodular_proof.tex
\section{A proof of Lemma~\ref{lem:submodular}}
\label{sec:submodular}

%See \S\ref{sec:background}, in particular \eqref{eqn:adjacent} and Definition~\ref{def:submodular} for notations.

\newtheorem*{lem:submodular}{Lemma~\ref{lem:submodular}}
\begin{lem:submodular}

A polytope $P \subseteq \RR^n$ is a generalized permutohedron if and only if there exists a submodular function $\omega : 2^{[n]} \rightarrow \RR$ with $\omega(\varnothing) = 0$ such that 
\begin{equation} 
\tag{\ref{eqn:polytope}}
P = \{x \in \RR^n : \sum_{i \in I} x_i \leq \omega(I) \text{ for each nonempty } I \subseteq [n], \text{ and } 
\sum_{i\in[n]}x_i = \omega([n])\}.
\end{equation}
A wall in the $S_n$ fan corresponding to $i \independent j \mid K$ is missing in the normal fan of $P$ defined by $\omega$ as above if and only if $\omega(Ki) + \omega(Kj) = \omega(Kij)+ \omega(K)$.
In particular, a coarsened $S_n$ fan is polytopal if and only if the corresponding semigraphoid is submodular.  
\end{lem:submodular}

Before proving the lemma, we first recall a general construction of the normal fan of a polytope from a halfspace description. See also~\cite[Theorem~9.5.6]{dLRS}.

Let $P = \{x \in \R^n : Ax + b \geq 0\}$ be a polytope, where $A$ is a $k \times n$ matrix and $b\in \RR^k$ is a column vector. Let us assume that $P$ is nonempty but is not necessarily full-dimensional.  Also assume that all inequalities are tight but possibly redundant.  In particular, if $a_i = a_j$, then $b_i = b_j$ where $(a_i,b_i)$ and $(a_j,b_j)$ are rows of $[A|b]$.
Let $C^*$ be the cone in $\R^{n} \times \R$ generated by the rows of the concatenated matrix $[A | b]$ and the vector $(\mathbf{0},1)$.  The row $(a_i,b_i)$ of $[A|b]$ is called a {\em lift} of the vector $a_i \in \RR^n$.
Since $P$ is bounded, for any $z\in \RR^n$ we have $Az \geq \mathbf{0} \Rightarrow z = \mathbf{0}$; otherwise $P$ would be unbounded in direction $z$.  
Then the rows of $A$ cannot be all contained in a halfspace $\{x : x \cdot z \geq 0\}$ for any nonzero $z$, so the rows of $A$ positively span $\RR^n$, and the cone $C^*$ projects surjectively onto~$\RR^n$.

The dual cone of $C^*$ is 
\begin{align}
  \label{eqn:dualCone}
  \begin{split}
C :=& \{v \in \RR^{n+1} : u\, \cdot\, v \geq 0 \text{ for all } u \in C^*\} \\
=& \{v \in \RR^{n+1}: [A | b] v \geq \mathbf{0} \text{ and } v_{n+1} \geq 0\} \\
=& \cone \{(x,1) : x \in P\},
\end{split}
\end{align}
where $\cone\{\cdot\}$ denotes the conical hull.  
 All nonzero vectors in the cone $C$ have positive last coordinates and hence all proper faces of $C^*$ are on the {\em lower hull} of $C^*$, that is, they have inward pointing normal vectors with positive last coordinate.  In particular, if a vector $(a,b)$ lies on the boundary of $C^*$, then $(a,b+\varepsilon)$ does not lie on the boundary of $C^*$ for any~\mbox{$\varepsilon > 0$}.  Since each inequality $a_i x + b_i \geq 0$ is assumed to be tight, there is a point $x^i \in P$ satisfying $a_i x^i + b_i = 0$, so the vector $(a_i, b_i)$ belongs to a proper face of $C^*$ that minimizes the linear functional $u \mapsto (x^i, 1) \cdot u$. 

 We claim that the projections of proper faces of $C^*$ onto $\R^n$ form the inner normal fan of $P$.  Let $p$ be a point in $P$ and consider the inner normal cone
\[
N_p := \{ c \in \RR^n : c \cdot p \leq c \cdot q \text{ for all } q \in P \}.
\]
We will show that $N_p$ is the projection of the following face of $C^*$ that minimizes the dot product with $(p,1)$:
\[
  \face_{(p,1)}(C^*) :=  \{u \in C^*: u \cdot (p,1) \leq u' \cdot (p,1) \text{ for all } u' \in C^*\}
   =  \{u \in C^*: u \cdot (p,1) = 0\}.
\]
Let $c \in N_p$. and let $u = (c, -c \,\cdot\, p)$.  Then $u \cdot (p,1) = c \,\cdot\, p - c \,\cdot\, p = 0$ and $u \cdot (q,1) = c \,\cdot\, q + u_{n+1} \geq c \,\cdot\, p + u_{n+1} = u \cdot (p,1) = 0$  for any $q\in P$; so we have found a vector $u \in \face_{(p,1)}(C^*)$ whose projection is $c$.  For the other inclusion, let $u \in \face_{(p,1)}(C^*)$.  Then $(u_1,\dots,u_n) \cdot p + u_{n+1} = 0$, while $(u_1,\dots,u_n) \cdot q + u_{n+1} \geq 0$ for all $q\in P$, and hence $(u_1,\dots,u_n) \cdot p \leq (u_1,\dots,u_n) \cdot q$.  Thus the projection $(u_1,\dots,u_n)$ belongs to $N_p$.  This shows that normal cones of $P$ are precisely the projections of proper faces of $C^*$.

In summary, a complete fan $\cF$ in $\RR^n$ is the normal fan of a polytope if and only if there exists a cone $K \subset \RR^n \times \RR$ with $(\mathbf{0},1) \in K$ whose proper faces project precisely onto cones of $\cF$. Given such a cone $K$, the desired polytope is obtained by slicing the dual cone $C$ of $K$ with the hyperplane $x_{n+1} = 1$ and projecting out $x_{n+1}$.  Given a polytope $P$, the desired cone $K$ can be obtained two ways: either as $K = (\cone\{(x,1) : x \in P\})^*$, or from an inequality description $Ax+b\geq 0$ of $P$ by lifting the rows of $A$ to height $b$ and taking the conical hull of these lifted rows together with the vector $(\mathbf{0},1)$.  Tightness of an inequality means that the corresponding vector is lifted to the boundary of~$K$. 

For any collection of vectors $\{(a_1,b_1),\dots,(a_k,b_k)\}$ that spans $K$ as a conical hull, we can consider the polytope $\{x \in \RR^n : a_i \cdot x + b_i \geq 0 \text{ for all } i\}$.  The arguments above show that this polytope is equal to $P$.
%The inequality $a_i \cdot x + b_i \geq 0$ is tight if $(a_i,b_i)$ is on the lower hull of $K$, that is, $b_i = \min\{b \in \RR : (a_i,b) \in K\}$, which is equal to $\min\{b \in \RR : a_i \cdot x + b \geq~0 \text{ for all } x \in~P\} = - \min\{a_i \cdot x: x \in P\}$. 

\begin{proof}[Proof of Lemma~\ref{lem:submodular}]

  Let $\cF$ be a coarsened $S_n$ fan which is the normal fan of a polytope $P$.  Every cone in $\cF$ contains a line in direction $(1,1,\dots,1)$ and is generated by this line together with some $0/1$ vectors. Then $\cF$ consists of the projection of faces of a cone in $\RR^n \times \RR$ generated by lifts of the $0/1$ vectors and $\pm(1,\dots,1)$.  By the paragraph preceding the proof, $P$ must have a tight halfspace description with normal vectors from the set $V = \{e_I \mid \emptyset \neq I \subseteq [n]\} \cup \{-e_{[n]}\}$. 
The ``right-hand sides'' of the inequalities give a lift $\omega : V \rightarrow \RR$ such that the proper faces of the cone $C^* = \cone\{(v,\omega(v)) \mid v \in V\}$ project precisely onto the cones of $\cF$ and every lifted vector is on the boundary of $C^*$.  Since all cones in $\cF$ contain the line $(1,1,\dots,1)$, we must have $\omega(e_{[n]}) = - \omega(- e_{[n]})$. Such lifts can be identified with functions on $2^{[n]}$ with value $0$ on $\emptyset$.  We will show that $\omega$ is submodular.

For any $I,J \subseteq [n]$, the vectors $e_I, e_{I \cap J}, e_{I \cup J}$ lie in a common cone in the $S_n$ fan. Since $\cF$ coarsens the $S_n$ fan, they also lie in a common cone  in $\cF$.  Similarly $e_J, e_{I \cap J}, e_{I \cup J}$ lie in a common cone of~$\cF$. 
First, consider the case when $e_I$ and $e_J$ are lifted to the same proper face of $C^*$.  Then this cone also contains $e_{I\cap J}$ and $e_{I \cup J}$.  Since we assumed that all lifted vectors lie on the boundary, hence a proper face, of $C^*$, and $\omega$ is linear on this face, we must have that $\omega(e_I) + \omega(e_J) = \omega(e_{I\cap J}) + \omega(e_{I \cup J})$.  

Now suppose that $e_{I}$ and $e_J$ are not lifted to the same proper face of $C^*$. Then $\omega$ is not linear on the vectors $e_I, e_J, e_{I\cap J}$, and $e_{I\cup J}$. We must then have that $\omega(e_I) + \omega(e_J) > \omega(e_{I\cap J}) + \omega(e_{I \cup J})$, because $\omega(e_I) + \omega(e_J) < \omega(e_{I\cap J}) + \omega(e_{I\cup J})$ would imply that 
\[(e_{I\cap J} + e_{I \cup J}, \omega(e_{I\cap J})+ \omega(e_{I \cup J}))
> (e_I + e_J, \omega(e_I)+\omega(e_J)),
\] 
contradicting the fact that $e_{I\cap J}$ and $e_{I \cup J}$ are lifted to the same cone in the lower hull of $C^*$.

For the converse, suppose $\omega$ is a submodular function on $2^{[n]}$ with $\omega(\emptyset)=0$ and consider the lift of $e_I$ to $\omega(I)$ for each $I \subseteq [n]$ and $-e_{[n]}$ to $-\omega([n])$.  Let $\cF$ be the projection of the lower hull of the lifted cone $C^*$.  The submodularity inequality $\omega(e_I) + \omega(e_J) \geq \omega(e_{I\cap J}) + \omega(e_{I \cup J})$ ensures that whenever $e_I$ and $e_J$ are lifted to the same cone in the lower hull of $C^*$, then so are $e_{I\cap J}$ and $e_{I \cup J}$.  In other words, whenever a cone of $\cF$ contains both, $e_I$ and $e_J$, then it must also contain both, $e_{I\cap J}$ and $e_{I\cup J}$, showing that $\cF$ is a coarsening of the $S_n$ fan.

Now suppose that the coarsened $S_n$ fan $\cF$ is polytopal and defined by a submodular function $\omega$ as above.  
Consider a wall of the $S_n$ fan corresponding to the adjacent permutations
\[
(a_1|\cdots|a_k|i|j| b_1|\cdots|b_{n-k-2}) \text{ and } (a_1|\cdots|a_k|j|i| b_1|\cdots|b_{n-k-2}),
\]
where $\{a_1,\dots,a_k\} = K$ and $\{b_1,\dots,b_{n-k-2}\} = [n]\backslash (K\cup\{i,j\})$.   
This wall is not contained in a wall of $\cF$ if and only if the two adjacent maximal cones are contained in the same cone of $\cF$.  In particular, this happens if and only if $e_{Ki}$ and $e_{Kj}$ are in the same cone where $K = \{a_1,\dots,a_k\}$. This is equivalent to having
\[\omega(Ki) + \omega(Kj) = \omega(Kij)+ \omega(K).\]

Let $P$ be the polytope defined by~(\ref{eqn:polytope}).  Its inner normal fan is obtained by lifting the rays $-e_I$ to height $\omega(I)$ for nonempty $I \subseteq [n]$ and $e_{[n]}$ to height $-\omega([n])$.  This is the negation of the fan $\cF$, which is obtained by lifting $e_I$ to $\omega(I)$ and $-e_{[n]}$ to $-\omega([n])$.  This shows that $\cF$ is the outer normal fan of $P$.
\end{proof}

%% file: dictionary.tex
\section{Dictionary}
\label{sec:dictionary}

The statements or data in each row are equivalent. \\

\begin{center}
\begin{tabular}{| m{0.23\textwidth} | m{0.3\textwidth}| m{0.38\textwidth} |}
\hline
{\bf CI relations} & {\bf Fans} & {\bf Polytopes} \\
\hline
\hline
CI relation \mbox{$i \independent j \mid K$} where $i,j \in [n]$, $K \subseteq [n]\setminus\{i,j\}$ & the set of walls in the $S_n$ fan of the form $\sigma | i\,j| \tau$ where $\sigma$ and $\tau$ are permutations of $K$ and $[n]\backslash Kij$ respectively & the set of edges of a permutohedron connecting two permutations of the form $\sigma | i | j| \tau$ and $\sigma | j | i | \tau$ where $\sigma$ and $\tau$ are permutations of $K$ and $[n]\backslash Kij$, respectively  \\
\hline
a collection of CI relations that satisfy the semigraphoid axioms & removing the walls in the $S_n$ fan corresponding to the independence relations gives a fan & the set of edges of the permutohedron corresponding to the independence relations satisfies the square and hexagon axioms~\cite{Morton_et_al} \\
\hline

a semigraphoid that arises from a submodular function & a coarsening of $S_n$ fan that is {\em polytopal} or {\em regular} & there is a generalized permutohedron that realizes contraction of edges in the permutohedron corresponding to the CI relations
\\
\hline
a union of dependence relations of a semigraphoid & a common refinement of fans & a Minkowski sum of polytopes (if the semigraphoid is submodular) \\
\hline
\end{tabular}
\end{center}